\documentclass[12pt]{article}

\usepackage{epsfig}
\usepackage{amsmath,amsfonts,amssymb,amsthm}
\usepackage{lineno}
\usepackage{color}
\usepackage{hyperref}
\usepackage{natbib}
\usepackage{multirow}
\usepackage{geometry}

\usepackage{makecell}
\usepackage[blocks]{authblk}

\newtheorem{thm}{Theorem}
\newtheorem{lem}[thm]{Lemma}

\allowdisplaybreaks

\begin{document}

\title{An Invariant Test for Equality of Two Large Scale Covariance Matrices} 
\author{Taehyeon Koo, Seonghun Cho, Johan Lim \thanks{All authors are at 
Department of Statistics, Seoul National University, Seoul, Korea. Taehyeon Koo and 
Seonghun Cho equally contribute to the paper. All correspondence are to Johan Lim, 
\texttt{E-mail:johanlim@snu.ac.kr}.}} 

\date{} 
\maketitle
\begin{abstract}
\noindent In this work, we are motivated by the recent work of \citet{Zhang:2019} and 
study a new invariant test for equality of two large scale covariance matrices. Two modified 
likelihood ratio tests (LRTs) by \citet{Zhang:2019} are based on the sum of log of eigenvalues (or 1- eigenvalues) of the 
Beta-matrix. However, as the dimension increases, many eigenvalues of the Beta-matrix 
are close to $0$ or $1$ and the modified LRTs are greatly influenced by them. In this work, instead, 
we consider the simple sum of the eigenvalues (of the 
Beta-matrix) and compute its asymptotic normality when all $n_1, n_2, p$ 
increase at the same rate. We numerically show that our test has higher power than two modified likelihood ratio tests 
by \citet{Zhang:2019} in all cases both we and they consider. 


\end{abstract}

\section{Introduction}\label{intro}

We revisit the test of equality (homogeneity) of two covariance matrices, which often allows us simplified procedures for many multivariate 
problems. Suppose we have samples $\mathbf{z}_{i}^{(l)}, i=1,2,\ldots,n_l,$ from a distribution with a mean vector $\mathbf{\mu}_{l}$ and 
covariance matrix $\Sigma_{l}$ for $l=1,2$. The hypothesis, which is of interest, is 
\begin{equation} \label{eqn:hypothesis} 
\mathcal{H}_0: \Sigma_1= \Sigma_2 \quad\mbox{versus} \quad \mathcal{H}_1: \Sigma_1\neq \Sigma_2. 
\end{equation}
As pointed out by \citet{Zhang:2019}, the history of the test draws back to 1930s and a huge number of works are followed in literature. In this paper, we do not aim to compete with all methods in the literature (see Chapter 10 of \citet{Anderson:2003} and references therein). Instead, we focus on a specific invariant test as an alternative to the modified likelihood ratio test (mLRT), which is recently suggested by \citet{Zhang:2019}. In Section 2, we compute the
asymptotic null distribution of the new test, when all $n_1,n_2,p$ increase at the same rate. In Section 3, we numerically show it has higher power than mLRT in all cases both we and they consider. In Section 4, we conclude the paper with some remarks.

\section{An alternative invariant statistic} 

We find $\mathbf{x}_i^{l}=\{x_{ij}^{(l)}\}$, for $i=1,2,\ldots,n_l, l=1,2$, where $x_{ij}^{l}$ are independent and identically distributed (IID) with mean zero and variance one, respectively, and satisfy 
\begin{equation} \nonumber 
\mathbf{z}_{j}^{(l)}=\Sigma_{l}^{1/2}\mathbf{x}_{j}^{(l)}+\mathbf{\mu}_{l}.
\end{equation} 
 Let $S_{n_{1}}$ and $S_{n_{2}}$ be the sample covariance matrix from each population. To build our test, 
we focus on the limiting distribution of eigenvalues of $\mathbf{B}_{n}=n_{1}S_{n_{1}}(n_{1}S_{n_{1}}+n_{2}S_{n_{2}})^{-1}$, 
named as the limiting spectral distribution (LSD) of $\mathbf{B}_{n}$. With the notations  
 \[
 y_{1}=p/n_{1}, y_{2}=p/n_{2},h=\sqrt{y_{1}+y_{2}-y_{1}y_{2}},\alpha_{n}=n_{2}/n_{1}, 
 \]
 the limiting spectral distribution of $\mathbf{B}_{n}$ is evaluated as (see \citet{Zhang:2019}) 
 \begin{equation} \nonumber
  \mathbf{F}_{\gamma_{1},\gamma_{2}}(x)=\frac{(\alpha+1)\sqrt{(x_{r}-x)(x-x_{l})}}{2\pi y_{1}x(1-x)}\delta_{x\in(x_{l},x_{r})},
  \end{equation}
 where $x_{l},x_{r}=\frac{y_{2}(h\mp y_{1})^2}{(y_{1}+y_{2})^2}, y_{1}\longrightarrow \gamma_{1}, y_{2}\longrightarrow \gamma_{2}$ and $\alpha_{n}\longrightarrow \alpha$ as $\min\{n_{1},n_{2},p\} \longrightarrow \infty$.

Let $G_{n_1,n_
2}(x)=p(\mathbf{F}^{\mathbf{B}_{n}}(x)-\mathbf{F}_{y_{1},y_{2}}(x)),$ where 
$\mathbf{F}^{\mathbf{B}_{n}}(x)$ is empirical spectral distribution (ESD) of $\mathbf{B}_{n}$ and $\mathbf{F}_{y_{1},y_{2}}(x)$ is the limit spectral distribution (LSD) of $\mathbf{B}_{n}$ with parameters $\alpha_{n},y_{1},y_{2}$ replacing $\alpha,\gamma_{1},\gamma_{2}$. Our 
main interest is in the limit distribution of  
 \begin{equation} \label{eqn:reg-lss} 
 \left( \int f_{1}(x)dG_{n_{1},n_{2}}(x),...,\int f_{k}(x)dG_{n_{1},n_{2}}(x) \right),
 \end{equation}
 where $f_{1},...,f_{k}$ are analytic functions on complex domain.

Suppose $F^{\{n_{1},n_{2}\}}(x), F^{\{y_{n_{1}},y_{n_{2}}\}}(x)$ are the ESD and LSD of the $F$-matrix $S_{n_{1}}S_{n_{2}}^{-1}$,
and 
\[
\tilde{G}_{n_{1},n_{2}}(x)=p(F^{\{n_{1},n_{2}\}}(x)- F^{\{y_{n_{1}},y_{n_{2}}\}}(x)).
\] 
Following \citet{Bai:2010}, the linear spectral statistic (LSS) of the $F$-matrix for functions $f_{1},...,f_{k}$ that is 
\[
\left(\int f_{1}(x)d\tilde{G}_{n_{1},n_{2}}(x),...,\int f_{k}(x)d\tilde{G}_{n_{1},n_{2}}(x)\right),
\]
under some regular conditions, converges weakly to a Gaussian vector $(X_{f_{1}},...,X_{f_{k}})'$ with means
 \begin{align*}
    {\rm E}X_{f_{i}}=\lim_{r\downarrow 1}\biggl( &\frac{1}{4\pi i}\oint_{|\xi|=1}f_{i}\biggl(\frac{|1+h\xi|^2}{(1-y_{2})^2}\biggr)\biggl[\frac{1}{\xi-r^{-1}}+\frac{1}{\xi+r^{-1}}-\frac{2}{\xi+\frac{y_{2}}{h}}\biggr]d\xi \\
     & +\frac{\Delta_{1}y_{1}(1-y_{2})^2}{2\pi i h^2}\oint_{|\xi|=1}f_{i}\biggl(\frac{|1+h\xi|^2}{(1-y_{2})^2}\biggr)\frac{1}{(\xi+\frac{y_{2}}{h})^3}d\xi \\
     &+\frac{\Delta_{2}(1-y_{2})}{4\pi i}\oint_{|\xi|=1}f_{i}\biggl(\frac{|1+h\xi|^2}{(1-y_{2})^2}\biggr)\frac{\xi^2-\frac{y_{2}}{h^2}}{(\xi+\frac{y_{2}}{h})^2}\biggl[\frac{1}{\xi^2-\frac{y_{2}}{h^2}}-\frac{2}{\xi+\frac{y_{2}}{h}}\biggr]d\xi\biggr)
 \end{align*}
 and the covariance matrix whose $(i,j)$ element is
 \begin{align*}
     {\rm Cov}(X_{f_{i}},X_{f_{j}})=\lim_{r\downarrow 1}\biggl(&-\frac{2}{4\pi^2}\oint_{|\xi_{1}|=1}\oint_{|\xi_{2}|=1}\frac{f_{i}(\frac{|1+h\xi_{1}|^2}{(1-y_{2})^2})f_{j}(\frac{|1+h\xi_{2}|^2}{(1-y_{2})^2}) }{(\xi_{1}-r\xi_{2})^2}d\xi_{1}d\xi_{2}\\
     &-\frac{\Delta_{1}y_{1}(1-y_{2})^2}{4\pi^2h^2}\oint_{|\xi_{1}|=1}\frac{f_{i}(\frac{|1+h\xi_{1}|^2}{(1-y_{2})^2})}{(\xi_{1}+\frac{y_{2}}{h})^2}d\xi_{1}\oint_{|\xi_{2}|=1}\frac{f_{j}(\frac{|1+h\xi_{2}|^2}{(1-y_{2})^2})}{(\xi_{2}+\frac{y_{2}}{h})^2}d\xi_{2} \\ 
     &-\frac{\Delta_{2}y_{2}(1-y_{2})^2}{4\pi^2h^2}\oint_{|\xi_{1}|=1}\frac{f_{i}(\frac{|1+h\xi_{1}|^2}{(1-y_{2})^2})}{(\xi_{1}+\frac{y_{2}}{h})^2}d\xi_{1}\oint_{|\xi_{2}|=1}\frac{f_{j}(\frac{|1+h\xi_{2}|^2}{(1-y_{2})^2})}{(\xi_{2}+\frac{y_{2}}{h})^2}d\xi_{2}\biggr).
 \end{align*}

If $\lambda$ is an eigenvalue of $S_{n_{1}}S_{n_{2}}^{-1}$, the eigenvalue of $\mathbf{B}_{n}=S_{n_{1}}(S_{n_{1}}+dS_{n_{2}})^{-1}$ 
 corresponds to $\lambda$ is $\frac{\lambda}{d+\lambda}$. Using this, we find that 
\begin{eqnarray} 
&& \left( \int f_{1}(x)dG_{n_{1},n_{2}}(x),...,\int f_{k}(x)dG_{n_{1},n_{2}}(x) \right ) \nonumber \\
 && \qquad = \left( \int f_{1}\left(\frac{x}{d+x}\right)d\tilde{G}_{n_{1},n_{2}}(x),...,\int f_{k}\left(\frac{x}{d+x} \right)d\tilde{G}_{n_{1},n_{2}}(x) \right ). \label{en:lss-reg} 
\end{eqnarray} 
In addition, we obtain the LSS of $\mathbf{B}_{n}$ from the above by substituting $d={n_{2}}/{n_{1}}$ in (\ref{en:lss-reg}).

The mLRT statistics in  \citet{Zhang:2019} are
\begin{equation}\nonumber
    \mathcal{L}= \sum_{\lambda_{i}^{\mathbf{B}_{n}}\in (0,1)}[c_{1}\text{log}\lambda_{i}^{\mathbf{B}_{n}}+c_{2}\text{log}(1-\lambda_{i}^{\mathbf{B}_{n}})], \quad
    \tilde{\mathcal{L}}  = \sum_{\lambda_{i}^{\mathbf{B}_{n}}\in (0,1)}{\log \lambda_{i}^{\mathbf{B}_{n}}},
\end{equation}
where $\lambda_{i}^{\mathbf{B}_{n}}$ denotes the $i$-th smallest eigenvalue of $\mathbf{B}_{n}$. In mLRT statistics, 
 the eigenvalues $0$ or $1$ are excluded for defining valid statistics. However, if ${p}/{n_{1}}$ (or ${p}/{n_{2}}$) is close to $1$, many eigenvalues are close to $0$ (or $1$). The eigenvalues 
 either close to $0$ or $1$ explains most part of the statistics. The mLRT statistics $\mathcal{L}$ and $\tilde{\mathcal{L}}$ 
 are sensitive to those and do not fully reflect  
 the information from other eigenvalues of $\mathbf{B}_{n}$.
 
To resolve this difficulty, we consider 
\begin{equation} \nonumber
 \mathcal{P} = \left( {\rm tr}(\mathbf{B}_{n}), {\rm tr}(\mathbf{I}_{p}-\mathbf{B}_{n}) \right)^{\top} = \left(\sum {\lambda_{i}^{\mathbf{B}_{n}}},\sum {(1-\lambda_{i}^{\mathbf{B}_{n}})}\right)^{\top}=\left(\mathcal{P}_{1},\mathcal{P}_{2}\right)^{\top}.
 \end{equation} 
To make above statistic meaningful, we modify it to   
\begin{equation} \label{eqn:newstat} 
\bigg(\sum_{\lambda_{i}^{\mathbf{B}_{n}}<1} {\lambda_{i}^{\mathbf{B}_{n}}},\sum_{\lambda_{i}^{\mathbf{B}_{n}}>0} {(1-\lambda_{i}^{\mathbf{B}_{n}})}\bigg)^{\top}.
\end{equation}

To get the asymptotic null distribution of the proposed statistic, we can find the mean and variance of LSS of $S_{n_{1}}(S_{n_{1}}+dS_{n_{2}})^{-1}$ by setting $k=2$,  $f_{1}(x) = x, f_{2}(x)=1-x$ in the formula above (\ref{en:lss-reg}). However, before we
proceed, we remark that  
\begin{eqnarray}
    &&\mathcal{P}_{2}-p\int(1-x)\frac{(\alpha_{n}+1)\sqrt{(x_{r}-x)(x-x_{l})}}{2\pi y_{1}x(1-x)} dx \nonumber\\
    &&=p-\mathcal{P}_{1}-p+p \int x\frac{(\alpha_{n}+1)\sqrt{(x_{r}-x)(x-x_{l})}}{2\pi y_{1}x(1-x)} dx\nonumber\\
    &&=- \left(\mathcal{P}_{1}-p\int x\frac{(\alpha_{n}+1)\sqrt{(x_{r}-x)(x-x_{l})}}{2\pi y_{1}x(1-x)} dx \right).\label{eq:minus} 
\end{eqnarray}
Thus, the LSS of $\mathcal{P}_{1}$ and $\mathcal{P}_{2}$ are opposite in their sign, and the covariance between LSS of $\mathcal{P}_{1},\mathcal{P}_{2}$ is the negative of asymptotic variance of $\mathcal{P}_{1}$ (or $\mathcal{P}_{2}$).

We now have our main results on the asymptotic null distributions of $\mathcal{P}_1$ and $\mathcal{P}_2$.
\begin{thm} \label{thm:1}
Suppose we assume (i) (moment assumptions)  ${\rm E} (x_{11}^{(1)})={\rm E}(x_{11}^{(2)})=0$, ${\rm E}(x_{11}^{(1)})^2={\rm E}(x_{11}^{(2)})^2=1$, and  ${\rm E}(x_{11}^{(1)})^4=\Delta_{1}+3 < \infty,$ ${\rm E}(x_{11}^{(2)})^4=\Delta_{2}+3 < \infty$,  and 
(ii) (dimensionality assumption) $\min (p,n_{1},n_{2}) \rightarrow \infty$, $y_{1} \neq 1$, $ y_{2} \neq 1$ and  $\lim p/(n_{1}+n_{2})<1$. Then we have
\begin{equation}\label{eqn:clt} 
    \mathcal{K}:=\frac{\mathcal{P}_{1}-p \ell_{n,1}-\mu_{n}}{\sigma_{n}} ~\overset{d}{\longrightarrow}~ Z  \quad \mbox{and} \quad 
    \mathcal{K}^{\prime}:=\frac{\mathcal{P}_{2}-p \ell_{n,2}+\mu_{n}}{\sigma_{n}} ~\overset{d}{\longrightarrow}~ Z,
\end{equation} 
where $Z \sim N(0,1)$,
\begin{equation} \nonumber 
  \ell_{n,1} = \frac{h^2\delta_{y_{2}>1}+y_{2}^2\delta_{y_{2}<1}}{y_{2}(y_{1}+y_{2})},~ \ell_{n,2}= \frac{h^2\delta_{y_{1}>1}+y_{1}^2\delta_{y_{1}<1}}{y_{1}(y_{1}+y_{2})}, 
 \end{equation} 
 and
 \begin{equation} \nonumber 
    \mu_{n} = -\frac{\Delta_{1}h^2y_{1}^2y_{2}^2}{(y_{1}+y_{2})^4}+\frac{\Delta_{2}h^2y_{1}^2y_{2}^2}{(y_{1}+y_{2})^4}, ~~ 
    \sigma_{n}^2 = \frac{2h^2y_{1}^2y_{2}^2}{(y_{1}+y_{2})^4}+(\Delta_{1}y_{1}+\Delta_{2}y_{2})\frac{h^4y_{1}^2y_{2}^2}{(y_{1}+y_{2})^6}.
\end{equation}
\end{thm}
\begin{proof} 
The proof is given in Appendix. 
\end{proof}

\section{Numerical study}\label{simul}

In this section, we numerically compare the powers of the proposed $\mathcal{K}$ to the modified LRTs $\mathcal{L}$ and $\tilde{\mathcal{L}}$ by  \citet{Zhang:2019} under various choices of sample sizes, dimensions, and distributions. In the study, we assume $\boldsymbol{\mu}_{l}=0$ for $l=1,2$ and set $\Sigma_{1}=(1+a/n_{1})\Sigma_{2}$, where $a$ is a constant.  We consider four cases following \citet{Zhang:2019}: 
\begin{itemize} 
\item Case 1:  $\mathbf{x}^{(1)}$ and $\mathbf{x}^{(2)}$ are from the standard normal distributed and  $\Sigma_{2}={\rm I}_{p}$;  
\item Case 2:  $\mathbf{x}^{(1)}$ and $\mathbf{x}^{(2)}$ are from the uniform distribution ${\rm U}(-\sqrt{3},\sqrt{3})$ and
  $\Sigma_{2}={\rm I}_{p}$; 
\item Case 3:  $\mathbf{x}^{(1)}$ and $\mathbf{x}^{(2)}$ are from the uniform distribution  ${\rm U}(-\sqrt{3},\sqrt{3})$ and
  $\Sigma_{2}={\rm diag}(p^2,{\bf 1}_{1 \times (p-1)})$;
 \item Case 4: $\mathbf{x}^{(1)}$ and $\mathbf{x}^{(2)}$ are from the uniform distribution  ${\rm U}(-\sqrt{3},\sqrt{3})$ and 
 $\Sigma_{2}=(0.5{\rm I}_{p}+0.5 {\bf 1}_{p \times p})$. 
\end{itemize} 
For each case, we consider $16$ choices of $(n_1,n_2,p)$, four cases are for each $(y_{1}>1, y_2>1)$,  $(y_1>1, y_2<1)$,$(y_1<1, y_2>1)$, 
$(y_1<1,y_2<1)$. Four choices are considered for $a$, $a=0,3,7,10$, where $a=0$ is the choice of the null hypothesis. 
 In all cases, we assume that forth moments of $\mathbf{x}^{(1)}$ and $\mathbf{x}^{(2)}$ are known. In case 1, $\Delta_{1}=\Delta_{2}=0$ and in case 2,3, and 4, $\Delta_{1}=\Delta_{2}=-1.2$.
For each 
combination, we generate $1000$ data sets and the powers (the sizes) are evaluated by counting the number of rejected data sets. 
The results are reported in  Tables \ref{tab:table1} to \ref{tab:table4}. In the tables, $\mathcal{L}$ and $\tilde{\mathcal{L}}$ stand for two modified LRTs 
by \citet{Zhang:2019} and $\mathcal{K}$ is our statistic in Theorem \ref{thm:1}. Moreover, Figures \ref{fig:power1} - \ref{fig:power4}
 display (empirical) powers with respect to different choices of $a$ for the Cases 1 - 4.  We find that, in all cases considered, the power of $\mathcal{K}$ is higher than both $\mathcal{L}$ and $\tilde{\mathcal{L}}$.

\begin{table}[htb!]
\resizebox{\textwidth}{!}{%
\footnotesize
\begin{tabular}{ccccccccc}
\hline
\multirow{2}{*}{\begin{tabular}[c]{@{}c@{}}\thead{(i)}\end{tabular}} & \multicolumn{4}{c}{$(n_{1},n_{2},p) = (45,40,50)$}       & \multicolumn{4}{c}{$(n_{1},n_{2},p) = (90,80,100)$}     \\ \cline{2-9} 
                  & a=0 & a=3 & a=7 & a=10 & a=0 & a=3 & a=7 & a=10      \\ \hline
                  $\mathcal{L}$& 0.062 & 0.056(0.046) & 0.058(0.044) & 0.070(0.060) & 0.050 & 0.051(0.058) & 0.059(0.063) & 0.069(0.072)     \\ \hline
                  $\tilde{\mathcal{L}}$& 0.061 & 0.109(0.094) & 0.335(0.303) & 0.555(0.506) & 0.045 & 0.102(0.110) & 0.336(0.350) & 0.610(0.625)        \\ \hline
                  $\mathcal{K}$& 0.062 & \textbf{0.150}(0.141) & \textbf{0.568}(0.535) & \textbf{0.835}(0.821) & 0.054 & \textbf{0.129}(0.117) & \textbf{0.603}(0.590) & \textbf{0.869}(0.865)        \\ \hline
\multirow{2}{*}{
} & \multicolumn{4}{c}{$(n_{1},n_{2},p) = (180,160,200)$}       & \multicolumn{4}{c}{$(n_{1},n_{2},p) = (360,320,400)$} \\ \cline{2-9} 
                  & a=0 & a=3 & a=7 & a=10 & a=0 & a=3 & a=7 & a=10      \\ \hline
                  $\mathcal{L}$& 0.041 & 0.043(0.058) & 0.063(0.085) & 0.070(0.091) & 0.056 & 0.041(0.036) & 0.053(0.043) & 0.080(0.074)      \\ \hline
                  $\tilde{\mathcal{L}}$& 0.047 & 0.109(0.119) & 0.348(0.356) & 0.620(0.639) & 0.059 & 0.105(0.092) & 0.352(0.330) & 0.632(0.606)      \\ \hline
                  $\mathcal{K}$& 0.049 & \textbf{0.175}(0.182) & \textbf{0.589}(0.597) & \textbf{0.882}(0.886) & 0.051 & \textbf{0.188}(0.188) & \textbf{0.601}(0.600) & \textbf{0.891}(0.891)      \\ \hline
\multirow{2}{*}{\begin{tabular}[c]{@{}c@{}}\thead{ (ii) }\end{tabular}} & \multicolumn{4}{c}{$(n_{1},n_{2},p) = (36,50,45)$} & \multicolumn{4}{c}{$(n_{1},n_{2},p) =(72,100,90)$}       \\ \cline{2-9} 
                  & a=0 & a=3 & a=7 & a=10 & a=0 & a=3 & a=7 & a=10      \\ \hline
                  $\mathcal{L}$& 0.066 & 0.071(0.056) & 0.110(0.092) & 0.227(0.188) & 0.056 & 0.065(0.187) & 0.131(0.122) & 0.211(0.195)        \\ \hline
                  $\tilde{\mathcal{L}}$&  0.049 & 0.169(0.170) & 0.519(0.523) & 0.806(0.809) & 0.042 & 0.158(0.187) & 0.602(0.639) & 0.858(0.874)        \\ \hline
                  $\mathcal{K}$& 0.060 & \textbf{0.216}(0.187) & \textbf{0.738}(0.718) & \textbf{0.957}(0.952) & 0.046 & \textbf{0.224}(0.228) & \textbf{0.807}(0.814) & \textbf{0.974}(0.976)        \\ \hline
\multirow{2}{*}{
} & \multicolumn{4}{c}{$(n_{1},n_{2},p) =(144,200,180)$} & \multicolumn{4}{c}{$(n_{1},n_{2},p) =(288,400,360)$}       \\ \cline{2-9} 
                  & a=0 & a=3 & a=7 & a=10 & a=0 & a=3 & a=7 & a=10      \\ \hline
                  $\mathcal{L}$&  0.057 & 0.055(0.051) & 0.130(0.117) & 0.209(0.195) & 0.043 & 0.060(0.070) & 0.142(0.158) & 0.225(0.242)      \\ \hline
                  $\tilde{\mathcal{L}}$& 0.065 & 0.174(0.142) & 0.614(0.577) & 0.864(0.831) & 0.046 & 0.167(0.175) & 0.626(0.635) & 0.897(0.901)      \\ \hline
                  $\mathcal{K}$ & 0.060 & \textbf{0.234}(0.209) & \textbf{0.819}(0.795) & \textbf{0.979}(0.975) & 0.055 & \textbf{0.251}(0.240) & \textbf{0.827}(0.821) & \textbf{0.982}(0.982)      \\ \hline
\multirow{2}{*}{\begin{tabular}[c]{@{}c@{}}\thead{(iii)}\end{tabular}} & \multicolumn{4}{c}{$(n_{1},n_{2},p) =(50,36,45)$} & \multicolumn{4}{c}{$(n_{1},n_{2},p) =(100,72,90)$}      \\ \cline{2-9} 
                  & a=0 & a=3 & a=7 & a=10 & a=0 & a=3 & a=7 & a=10      \\ \hline
                  $\mathcal{L}$& 0.061 & 0.056(0.044) & 0.095(0.086) & 0.116(0.103) & 0.061 & 0.074(0.046) & 0.084(0.056) & 0.099(0.068)        \\ \hline
                  $\tilde{\mathcal{L}}$& 0.060 & 0.083(0.063) & 0.291(0.257) & 0.511(0.461) & 0.058 & 0.107(0.097) & 0.293(0.276) & 0.516(0.498)       \\ \hline
                  $\mathcal{K}$& 0.060 & \textbf{0.140}(0.123) & \textbf{0.494}(0.473) & \textbf{0.780}(0.758) & 0.057 & \textbf{0.153}(0.144) & \textbf{0.505}(0.489) & \textbf{0.822}(0.808)       \\ \hline
\multirow{2}{*}{
} & \multicolumn{4}{c}{$(n_{1},n_{2},p) =(200,144,180)$} & \multicolumn{4}{c}{$(n_{1},n_{2},p) =(400,288,360)$}       \\  \cline{2-9} 
                  & a=0 & a=3 & a=7 & a=10 & a=0 & a=3 & a=7 & a=10     \\ \hline
                  $\mathcal{L}$& 0.041 & 0.051(0.058) & 0.083(0.096) & 0.142(0.155) & 0.053 & 0.063(0.062) & 0.071(0.071) & 0.123(0.119)      \\ \hline
                  $\tilde{\mathcal{L}}$&  0.042 & 0.086(0.094) & 0.286(0.305) & 0.563(0.579) & 0.054 & 0.095(0.093) & 0.318(0.312) & 0.535(0.525)     \\ \hline
                  $\mathcal{K}$& 0.033 & \textbf{0.138}(0.176)  & \textbf{0.505}(0.565) & \textbf{0.825}(0.849) & 0.038 & \textbf{0.145}(0.171) & \textbf{0.576}(0.615) & \textbf{0.822}(0.848)      \\ \hline
\multirow{2}{*}{\begin{tabular}[c]{@{}c@{}}\thead{ (iv) }\end{tabular}} & \multicolumn{4}{c}{$(n_{1},n_{2},p) =(50,50,45)$} & \multicolumn{4}{c}{$(n_{1},n_{2},p) =(100,100,90)$}       \\ \cline{2-9} 
                  & a=0 & a=3 & a=7 & a=10 & a=0 & a=3 & a=7 & a=10      \\ \hline
                  $\mathcal{L}$& 0.061 & 0.059(0.042) & 0.047(0.036) & 0.061(0.043) & 0.063 & 0.051(0.046) & 0.387(0.384) & 0.053(0.047)       \\ \hline
                  $\tilde{\mathcal{L}}$& 0.059 & 0.108(0.090) & 0.345(0.314) & 0.592(0.555) & 0.053 & 0.116(0.113) & 0.387(0.384) & 0.636(0.632)        \\ \hline
                  $\mathcal{K}$& 0.048 & \textbf{0.146}(0.154) & \textbf{0.617}(0.624) & \textbf{0.892}(0.896) & 0.055 & \textbf{0.202}(0.189) & \textbf{0.666}(0.651) & \textbf{0.907}(0.902)      \\ \hline
\multirow{2}{*}{
} & \multicolumn{4}{c}{$(n_{1},n_{2},p) =(200,200,180)$} & \multicolumn{4}{c}{$(n_{1},n_{2},p) =(400,400,360)$}       \\  \cline{2-9} 
                  & a=0 & a=3 & a=7 & a=10 & a=0 & a=3 & a=7 & a=10      \\ \hline
                  $\mathcal{L}$& 0.042 & 0.060(0.069) & 0.059(0.064) & 0.047(0.056) & 0.044 & 0.053(0.056) & 0.058(0.064) & 0.49(0.054)      \\ \hline
                  $\tilde{\mathcal{L}}$& 0.049 & 0.112(0.114) & 0.403(0.406) & 0.656(0.658) & 0.046 & 0.119(0.130) & 0.391(0.408) & 0.686(0.700)      \\ \hline
                  $\mathcal{K}$& 0.049 & \textbf{0.187}(0.187) & \textbf{0.675}(0.678) & \textbf{0.917}(0.917) & 0.055 & \textbf{0.195}(0.187) & \textbf{0.694}(0.684) & \textbf{0.931}(0.927)      \\ \hline
\end{tabular}%
}
\caption{Case 1: Empirical sizes and powers. The results are based on the 5\% significance level. The power number in small bracket indicates the size corrected power based empirical 95\% quantile not asymptotic. (i) indicates the case of $y_{1}>1,y_{2}>1$, (ii) does $y_{1}>1,y_{2}<1$, (iii) does $y_{1}<1,y_{2}>1$, and (iv) does $y_{1}<1,y_{2}<1$.}
\label{tab:table1}
\end{table}

\begin{table}[htb!]
\resizebox{\textwidth}{!}{%
\footnotesize
\begin{tabular}{ccccccccc}
\hline
\multirow{2}{*}{\begin{tabular}[c]{@{}c@{}}\thead{ (i)}\end{tabular}} & \multicolumn{4}{c}{$(n_{1},n_{2},p)=(45,40,50)$}       & \multicolumn{4}{c}{$(n_{1},n_{2},p)=(90,80,100)$}       \\  \cline{2-9} 
                  & a=0 & a=3 & a=7 & a=10 & a=0 & a=3 & a=7 & a=10      \\ \hline
                  $\mathcal{L}$& 0.060 & 0.055(0.041) & 0.067(0.051) & 0.061(0.047) & 0.054 & 0.052(0.047) & 0.084(0.075) & 0.078(0.072)        \\ \hline
                  $\tilde{\mathcal{L}}$& 0.061 & 0.136(0.120) & 0.371(0.328) & 0.602(0.572) & 0.064 & 0.129(0.105) & 0.396(0.342) & 0.666(0.623)       \\ \hline
                  $\mathcal{K}$& 0.054 & \textbf{0.203}(0.190) & \textbf{0.700}(0.686) & \textbf{0.917}(0.913) & 0.052 & \textbf{0.207}(0.207) & \textbf{0.732}(0.730) & \textbf{0.951}(0.950)        \\ \hline
\multirow{2}{*}{
} & \multicolumn{4}{c}{$(n_{1},n_{2},p)=(180,160,200)$}       & \multicolumn{4}{c}{$(n_{1},n_{2},p)=(360,320,400)$}       \\  \cline{2-9} 
                  & a=0 & a=3 & a=7 & a=10 & a=0 & a=3 & a=7 & a=10      \\ \hline
                  $\mathcal{L}$& 0.045 & 0.050(0.068) & 0.064(0.081) & 0.098(0.113) & 0.067 & 0.054(0.040) & 0.072(0.056) & 0.083(0.062)      \\ \hline
                  $\tilde{\mathcal{L}}$& 0.047 & 0.122(0.130) & 0.399(0.410) & 0.692(0.712) & 0.061 & 0.114(0.091) & 0.392(0.355) & 0.711(0.683)      \\ \hline
                  $\mathcal{K}$& 0.055 & \textbf{0.199}(0.194) & \textbf{0.742}(0.731) & \textbf{0.953}(0.943) & 0.050 & \textbf{0.212}(0.214) & \textbf{0.754}(0.758) & \textbf{0.964}(0.965)      \\ \hline
\multirow{2}{*}{\begin{tabular}[c]{@{}c@{}}\thead{(ii)}\end{tabular}} & \multicolumn{4}{c}{$(n_{1},n_{2},p)=(36,50,45)$} & \multicolumn{4}{c}{$(n_{1},n_{2},p)=(72,100,90)$}       \\  \cline{2-9} 
                  & a=0 & a=3 & a=7 & a=10 & a=0 & a=3 & a=7 & a=10      \\ \hline
                  $\mathcal{L}$& 0.060 & 0.059(0.043) & 0.132(0.116) & 0.160(0.133) & 0.066 & 0.057(0.044) & 0.100(0.078) & 0.169(0.129)        \\ \hline
                  $\tilde{\mathcal{L}}$& 0.047 & 0.177(0.184) & 0.646(0.660) & 0.857(0.858) & 0.044 & 0.168(0.201) & 0.681(0.713) & 0.919(0.933)        \\ \hline
                  $\mathcal{K}$& 0.051 & \textbf{0.283}(0.283) & \textbf{0.890}(0.890) & \textbf{0.991}(0.991) & 0.048 & \textbf{0.308}(0.327) & \textbf{0.903}(0.911) & \textbf{0.999}(0.999)        \\ \hline
\multirow{2}{*}{
} & \multicolumn{4}{c}{$(n_{1},n_{2},p)=(144,200,180)$} & \multicolumn{4}{c}{$(n_{1},n_{2},p)=(288,400,360)$}       \\  \cline{2-9} 
                  & a=0 & a=3 & a=7 & a=10 & a=0 & a=3 & a=7 & a=10      \\ \hline
                  $\mathcal{L}$& 0.056 & 0.055(0.049) & 0.110(0.096) & 0.193(0.175) & 0.063 & 0.062(0.059) & 0.095(0.087) & 0.164(0.148)      \\ \hline
                  $\tilde{\mathcal{L}}$& 0.062 & 0.183(0.176) & 0.684(0.665) & 0.932(0.919) & 0.037 & 0.165(0.191) & 0.692(0.731) & 0.937(0.944)      \\ \hline
                  $\mathcal{K}$ & 0.055 & \textbf{0.321}(0.314) & \textbf{0.914}(0.907) & \textbf{0.998}(0.997) & 0.050 & \textbf{0.310}(0.310) & \textbf{0.922}(0.924) & \textbf{0.999}(0.999)      \\ \hline
\multirow{2}{*}{\begin{tabular}[c]{@{}c@{}}\thead{(iii)}\end{tabular}} & \multicolumn{4}{c}{$(n_{1},n_{2},p)=(50,36,45)$} & \multicolumn{4}{c}{$(n_{1},n_{2},p)=(100,72,90)$}       \\  \cline{2-9} 
                  & a=0 & a=3 & a=7 & a=10 & a=0 & a=3 & a=7 & a=10      \\ \hline
                  $T$& 0.057 & 0.066(0.057) & 0.096(0.088) & 0.119(0.104) & 0.053 & 0.053(0.052) & 0.104(0.103) & 0.129(0.127)        \\ \hline
                  $\tilde{T}$& 0.057 & 0.106(0.099) & 0.331(0.316) & 0.559(0.543) & 0.059 & 0.090(0.079) & 0.352(0.314) & 0.601(0.552)        \\ \hline
                  $\mathcal{K}$& 0.048 & \textbf{0.162}(0.176) & \textbf{0.618}(0.640) & \textbf{0.899}(0.909) & 0.040 & \textbf{0.181}(0.199) & \textbf{0.667}(0.688) & \textbf{0.927}(0.933)      \\ \hline
\multirow{2}{*}{
} & \multicolumn{4}{c}{$(n_{1},n_{2},p)=(200,144,180)$} & \multicolumn{4}{c}{$(n_{1},n_{2},p)=(400,288,360)$}       \\ \cline{2-9} 
                  & a=0 & a=3 & a=7 & a=10 & a=0 & a=3 & a=7 & a=10     \\ \hline
                  $\mathcal{L}$& 0.053 & 0.066(0.066) & 0.089(0.088) & 0.129(0.127) & 0.052 & 0.057(0.057) & 0.091(0.091) & 0.128(0.128)     \\ \hline
                  $\tilde{\mathcal{L}}$&  0.052 & 0.124(0.111) & 0.350(0.334) & 0.600(0.586) & 0.046 & 0.104(0.109) & 0.369(0.380) & 0.605(0.618)      \\ \hline
                  $\mathcal{K}$&  0.039 & \textbf{0.190}(0.214) & \textbf{0.685}(0.718) & \textbf{0.928}(0.941) & 0.049 & \textbf{0.199}(0.199) & \textbf{0.689}(0.691) & \textbf{0.926}(0.928)      \\ \hline
\multirow{2}{*}{\begin{tabular}[c]{@{}c@{}}\thead{(iv)}\end{tabular}} & \multicolumn{4}{c}{$(n_{1},n_{2},p)=(50,50,45)$} & \multicolumn{4}{c}{$(n_{1},n_{2},p)=(100,100,90)$}       \\ \cline{2-9} 
                  & a=0 & a=3 & a=7 & a=10 & a=0 & a=3 & a=7 & a=10      \\ \hline
                  $\mathcal{L}$& 0.040 & 0.051(0.057) & 0.049(0.057) & 0.050(0.056) & 0.048 & 0.053(0.054) & 0.061(0.063) & 0.056(0.059)      \\ \hline
                  $\tilde{\mathcal{L}}$&  0.046 & 0.133(0.133) & 0.409(0.435) & 0.656(0.677) & 0.046 & 0.108(0.113) & 0.421(0.434) & 0.707(0.713)     \\ \hline
                  $\mathcal{K}$& 0.052 & \textbf{0.254}(0.250) & \textbf{0.807}(0.806) & \textbf{0.970}(0.970) & 0.043 & \textbf{0.222}(0.240) & \textbf{0.814}(0.840) & \textbf{0.981}(0.983)        \\ \hline
\multirow{2}{*}{
} & \multicolumn{4}{c}{$(n_{1},n_{2},p)=(200,200,180)$} & \multicolumn{4}{c}{$(n_{1},n_{2},p)=(400,400,360)$}       \\ \cline{2-9} 
                  & a=0 & a=3 & a=7 & a=10 & a=0 & a=3 & a=7 & a=10      \\ \hline
                  $\mathcal{L}$& 0.051 & 0.052(0.050) & 0.055(0.054) & 0.050(0.050) & 0.050 & 0.061(0.062) & 0.050(0.050) & 0.052(0.052)      \\ \hline
                  $\tilde{\mathcal{L}}$&  0.050 & 0.130(0.133) & 0.458(0.460) & 0.720(0.722) & 0.040 & 0.117(0.127) & 0.448(0.462) & 0.741(0.754)     \\ \hline
                  $\mathcal{K}$& 0.047 & \textbf{0.259}(0.268) & \textbf{0.833}(0.846) & \textbf{0.984}(0.986) & 0.058 & \textbf{0.252}(0.241) & \textbf{0.836}(0.827) & \textbf{0.994}(0.992)      \\ \hline
\end{tabular}%
}
\caption{Case 2: Empirical sizes and powers. The results are based on the 5\% significance level. The power number in small bracket indicates the size corrected power based empirical 95\% quantile not asymptotic. (i) indicates the case of $y_{1}>1,y_{2}>1$, (ii) does $y_{1}>1,y_{2}<1$, (iii) does $y_{1}<1,y_{2}>1$, and (iv) does $y_{1}<1,y_{2}<1$.}
\label{tab:table2}
\end{table}

\begin{table}[htb!]
\resizebox{\textwidth}{!}{%
\footnotesize
\begin{tabular}{ccccccccc}
\hline
\multirow{2}{*}{\begin{tabular}[c]{@{}c@{}}\thead{(i)}\end{tabular}} & \multicolumn{4}{c}{$(n_{1},n_{2},p)=(45,40,50)$}       & \multicolumn{4}{c}{$(n_{1},n_{2},p)=(90,80,100)$}       \\ \cline{2-9} 
                  & a=0 & a=3 & a=7 & a=10 & a=0 & a=3 & a=7 & a=10      \\ \hline
                  $\mathcal{L}$&  0.055 & 0.050(0.039) & 0.067(0.061) & 0.066(0.057) & 0.053 & 0.061(0.059) & 0.061(0.055) & 0.093(0.088)        \\ \hline
                  $\tilde{\mathcal{L}}$& 0.060 & 0.112(0.102) & 0.368(0.354) & 0.586(0.570) & 0.050 & 0.124(0.131) & 0.372(0.379) & 0.645(0.651)        \\ \hline
                  $\mathcal{K}$& 0.049 & \textbf{0.185}(0.194) & \textbf{0.717}(0.740) & \textbf{0.915}(0.923) & 0.054 & \textbf{0.220}(0.215) & \textbf{0.711}(0.706) & \textbf{0.947}(0.945)        \\ \hline
\multirow{2}{*}{
} & \multicolumn{4}{c}{$(n_{1},n_{2},p)=(180,160,200)$}       & \multicolumn{4}{c}{$(n_{1},n_{2},p)=(360,320,400)$}       \\  \cline{2-9} 
                  & a=0 & a=3 & a=7 & a=10 & a=0 & a=3 & a=7 & a=10      \\ \hline
                  $\mathcal{L}$& 0.035 & 0.075(0.090) & 0.070(0.089) & 0.087(0.105) & 0.039 & 0.057(0.072) & 0.066(0.076) & 0.084(0.096)     \\ \hline
                  $\tilde{\mathcal{L}}$& 0.043 & 0.113(0.122) & 0.408(0.423) & 0.693(0.705) & 0.054 & 0.111(0.102) & 0.394(0.378) & 0.697(0.678)      \\ \hline
                  $\mathcal{K}$& 0.050 & \textbf{0.186}(0.186) & \textbf{0.738}(0.738) & \textbf{0.955}(0.955) & 0.041 & \textbf{0.199}(0.236) & \textbf{0.741}(0.765) & \textbf{0.961}(0.967)      \\ \hline
\multirow{2}{*}{\begin{tabular}[c]{@{}c@{}}\thead{(ii)}\end{tabular}} & \multicolumn{4}{c}{$(n_{1},n_{2},p)=(36,50,45)$} & \multicolumn{4}{c}{$(n_{1},n_{2},p)=(72,100,90)$}       \\  \cline{2-9} 
                  & a=0 & a=3 & a=7 & a=10 & a=0 & a=3 & a=7 & a=10      \\ \hline
                  $\mathcal{L}$&  0.070 & 0.054(0.043) & 0.104(0.081) & 0.165(0.116) & 0.057 & 0.048(0.043) & 0.108(0.099) & 0.175(0.163)        \\ \hline
                  $\tilde{\mathcal{L}}$& 0.056 & 0.170(0.159) & 0.631(0.619) & 0.859(0.853) & 0.050 & 0.185(0.189) & 0.654(0.660) & 0.921(0.921)        \\ \hline
                  $\mathcal{K}$& 0.056 & \textbf{0.280}(0.268) & \textbf{0.880}(0.869) & \textbf{0.979}(0.979) & 0.052 & \textbf{0.309}(0.298) & \textbf{0.902}(0.900) & \textbf{0.992}(0.992)        \\ \hline
\multirow{2}{*}{
} & \multicolumn{4}{c}{$(n_{1},n_{2},p)=(144,200,180)$} & \multicolumn{4}{c}{$(n_{1},n_{2},p)=(288,400,360)$}       \\  \cline{2-9} 
                  & a=0 & a=3 & a=7 & a=10 & a=0 & a=3 & a=7 & a=10      \\ \hline
                  $\mathcal{L}$& 0.054 & 0.062(0.054) & 0.095(0.091) & 0.179(0.169) & 0.054 & 0.058(0.055) & 0.105(0.103) & 0.160(0.156)      \\ \hline
                  $\tilde{\mathcal{L}}$& 0.044 & 0.183(0.192) & 0.691(0.700) & 0.927(0.928) & 0.039 & 0.182(0.199) & 0.669(0.693) & 0.934(0.947)      \\ \hline
                  $\mathcal{K}$ &  0.046 & \textbf{0.295}(0.307) & \textbf{0.902}(0.915) & \textbf{0.999}(0.999) & 0.047 & \textbf{0.328}(0.346) & \textbf{0.923}(0.935) & \textbf{0.998}(0.998)      \\ \hline
\multirow{2}{*}{\begin{tabular}[c]{@{}c@{}}\thead{(iii)}\end{tabular}} & \multicolumn{4}{c}{$(n_{1},n_{2},p)=(50,36,45)$} & \multicolumn{4}{c}{$(n_{1},n_{2},p)=(100,72,90)$}       \\  \cline{2-9} 
                  & a=0 & a=3 & a=7 & a=10 & a=0 & a=3 & a=7 & a=10      \\ \hline
                  $\mathcal{L}$& 0.060 & 0.056(0.051) & 0.088(0.079) & 0.116(0.107) & 0.050 & 0.076(0.076) & 0.109(0.110) & 0.116(0.116)        \\ \hline
                  $\tilde{\mathcal{L}}$&0.054 & 0.092(0.074) & 0.324(0.301) & 0.558(0.541) & 0.045 & 0.112(0.132) & 0.354(0.382) & 0.592(0.625)        \\ \hline
                  $\mathcal{K}$& 0.050 & \textbf{0.155}(0.155) & \textbf{0.645}(0.645) & \textbf{0.903}(0.903) & 0.044 & \textbf{0.167}(0.187) & \textbf{0.659}(0.687) & \textbf{0.922}(0.935)        \\ \hline
\multirow{2}{*}{
} & \multicolumn{4}{c}{$(n_{1},n_{2},p)=(200,144,180)$} & \multicolumn{4}{c}{$(n_{1},n_{2},p)=(400,288,360)$}       \\  \cline{2-9} 
                  & a=0 & a=3 & a=7 & a=10 & a=0 & a=3 & a=7 & a=10     \\ \hline
                  $\mathcal{L}$& 0.061 & 0.053(0.046) & 0.079(0.065) & 0.126(0.113) & 0.043 & 0.057(0.060) & 0.082(0.091) & 0.136(0.147)      \\ \hline
                  $\tilde{\mathcal{L}}$& 0.065 & 0.107(0.085) & 0.334(0.299) & 0.605(0.565) & 0.043 & 0.096(0.109) & 0.327(0.353) & 0.626(0.642)      \\ \hline
                  $\mathcal{K}$& 0.047 & \textbf{0.176}(0.182) & \textbf{0.676}(0.692) & \textbf{0.923}(0.927) & 0.045 & \textbf{0.175}(0.185) & \textbf{0.673}(0.680) & \textbf{0.949}(0.949)      \\ \hline
\multirow{2}{*}{\begin{tabular}[c]{@{}c@{}}\thead{(iv)}\end{tabular}} & \multicolumn{4}{c}{$(n_{1},n_{2},p)=(50,50,45)$} & \multicolumn{4}{c}{$(n_{1},n_{2},p)=(100,100,90)$}       \\  \cline{2-9} 
                  & a=0 & a=3 & a=7 & a=10 & a=0 & a=3 & a=7 & a=10      \\ \hline
                  $\mathcal{L}$& 0.065 & 0.043(0.035) & 0.056(0.046) & 0.066(0.056) & 0.060 & 0.046(0.042) & 0.058(0.053) & 0.049(0.046)        \\ \hline
                  $\tilde{\mathcal{L}}$& 0.049 & 0.110(0.112) & 0.420(0.425) & 0.658(0.665) & 0.059 & 0.135(0.113) & 0.419(0.386) & 0.700(0.676)        \\ \hline
                  $\mathcal{K}$& 0.044 & \textbf{0.223}(0.236) & \textbf{0.780}(0.793) & \textbf{0.967}(0.974) & 0.050 & \textbf{0.239}(0.242) & \textbf{0.814}(0.817) & \textbf{0.979}(0.979)        \\ \hline
\multirow{2}{*}{
} & \multicolumn{4}{c}{$(n_{1},n_{2},p)=(200,200,180)$} & \multicolumn{4}{c}{$(n_{1},n_{2},p)(400,400,360)$}       \\  \cline{2-9} 
                  & a=0 & a=3 & a=7 & a=10 & a=0 & a=3 & a=7 & a=10      \\ \hline
                  $\mathcal{L}$& 0.049 & 0.045(0.046) & 0.046(0.046) & 0.054(0.057) & 0.041 & 0.057(0.068) & 0.047(0.064) & 0.061(0.073)      \\ \hline
                  $\tilde{\mathcal{L}}$&0.052 & 0.117(0.107) & 0.419(0.413) & 0.744(0.736) & 0.041 & 0.126(0.132) & 0.473(0.491) & 0.752(0.760)      \\ \hline
                  $\mathcal{K}$& 0.050 & \textbf{0.257}(0.265) & \textbf{0.854}(0.857) & \textbf{0.989}(0.989) & 0.047 & \textbf{0.240}(0.250) & \textbf{0.864}(0.868) & \textbf{0.993}(0.993)      \\ \hline
\end{tabular}%
}
\caption{Case 3: Empirical sizes and powers. The results are based on the 5\% significance level. The power number in small bracket indicates the size corrected power based empirical 95\% quantile not asymptotic. (i) indicates the case of $y_{1}>1,y_{2}>1$, (ii) does $y_{1}>1,y_{2}<1$, (iii) does $y_{1}<1,y_{2}>1$, and (iv) does $y_{1}<1,y_{2}<1$.}
\label{tab:table3}
\end{table}

\begin{table}[htb!]
\resizebox{\textwidth}{!}{%
\footnotesize
\begin{tabular}{ccccccccc}
\hline
\multirow{2}{*}{\begin{tabular}[c]{@{}c@{}}\thead{(i)}\end{tabular}} & \multicolumn{4}{c}{$(n_{1},n_{2},p)=(45,40,50)$}       & \multicolumn{4}{c}{$(n_{1},n_{2},p)=(90,80,100)$}       \\  \cline{2-9} 
                  & a=0 & a=3 & a=7 & a=10 & a=0 & a=3 & a=7 & a=10      \\ \hline
                  $\mathcal{L}$& 0.055 & 0.069(0.060) & 0.074(0.071) & 0.083(0.078) & 0.058 & 0.061(0.054) & 0.064(0.055) & 0.075(0.065)        \\ \hline
                  $\tilde{\mathcal{L}}$& 0.057 & 0.104(0.096) & 0.359(0.347) & 0.599(0.581) & 0.054 & 0.111(0.109) & 0.400(0.396) & 0.663(0.659)        \\ \hline
                  $\mathcal{K}$& 0.055 & \textbf{0.175}(0.166) & \textbf{0.705}(0.686) & \textbf{0.937}(0.933) & 0.058 & \textbf{0.179}(0.166) & \textbf{0.743}(0.728) & \textbf{0.956}(0.951)        \\ \hline
\multirow{2}{*}{
} & \multicolumn{4}{c}{$(n_{1},n_{2},p)=(180,160,200)$}       & \multicolumn{4}{c}{$(n_{1},n_{2},p)=(360,320,400)$}       \\  \cline{2-9} 
                  & a=0 & a=3 & a=7 & a=10 & a=0 & a=3 & a=7 & a=10      \\ \hline
                  $\mathcal{L}$& 0.054 & 0.061(0.055) & 0.058(0.054) & 0.091(0.083) & 0.042 & 0.044(0.048) & 0.081(0.099) & 0.085(0.097)      \\ \hline
                  $\tilde{\mathcal{L}}$& 0.059 & 0.118(0.115) & 0.374(0.361) & 0.665(0.652) & 0.046 & 0.122(0.142) & 0.387(0.423) &  0.687(0.722)      \\ \hline
                  $\mathcal{K}$& 0.066 & \textbf{0.195}(0.150) & \textbf{0.731}(0.690) & \textbf{0.953}(0.943) & 0.046 & \textbf{0.219}(0.230) & \textbf{0.735}(0.745) & \textbf{0.954}(0.956)      \\ \hline
\multirow{2}{*}{\begin{tabular}[c]{@{}c@{}}\thead{(ii)}\end{tabular}} & \multicolumn{4}{c}{$(n_{1},n_{2},p)=(36,50,45)$} & \multicolumn{4}{c}{$(n_{1},n_{2},p)=(72,100,90)$}       \\  \cline{2-9} 
                  & a=0 & a=3 & a=7 & a=10 & a=0 & a=3 & a=7 & a=10      \\ \hline
                  $\mathcal{L}$& 0.046 & 0.058(0.058) & 0.119(0.125) & 0.174(0.179) & 0.064 & 0.062(0.044) & 0.138(0.109) & 0.162(0.128)        \\ \hline
                  $\tilde{\mathcal{L}}$& 0.037 & 0.182(0.237) & 0.619(0.742) & 0.872(0.902) & 0.046 & 0.178(0.189) & 0.673(0.685) & 0.897(0.904)        \\ \hline
                  $\mathcal{K}$& 0.045 & \textbf{0.302}(0.315) & \textbf{0.869}(0.874) & \textbf{0.993}(0.993) & 0.052 & \textbf{0.295}(0.289) & \textbf{0.899}(0.898) & \textbf{0.995}(0.995)        \\ \hline
\multirow{2}{*}{
} & \multicolumn{4}{c}{$(n_{1},n_{2},p)=(144,200,180)$} & \multicolumn{4}{c}{$(n_{1},n_{2},p)=(288,400,360)$}       \\  \cline{2-9} 
                  & a=0 & a=3 & a=7 & a=10 & a=0 & a=3 & a=7 & a=10      \\ \hline
                  $\mathcal{L}$& 0.063 & 0.048(0.045) & 0.108(0.096) & 0.163(0.147) & 0.061 & 0.063(0.059) & 0.106(0.095) & 0.165(0.144)      \\ \hline
                  $\tilde{\mathcal{L}}$& 0.044 & 0.176(0.195) & 0.672(0.694) & 0.936(0.941) & 0.053 & 0.215(0.211) & 0.709(0.701) & 0.933(0.931)      \\ \hline
                  $\mathcal{K}$ & 0.049 & \textbf{0.322}(0.324) & \textbf{0.922}(0.922) & \textbf{0.997}(0.997) & 0.045 & \textbf{0.349}(0.378) & \textbf{0.921}(0.927) & \textbf{0.996}(0.997)      \\ \hline
\multirow{2}{*}{\begin{tabular}[c]{@{}c@{}}\thead{(iii)}\end{tabular}} & \multicolumn{4}{c}{$(n_{1},n_{2},p)=(50,36,45)$} & \multicolumn{4}{c}{$(n_{1},n_{2},p)=(100,72,90)$}       \\  \cline{2-9} 
                  & a=0 & a=3 & a=7 & a=10 & a=0 & a=3 & a=7 & a=10      \\ \hline
                  $\mathcal{L}$& 0.058 & 0.073(0.068) & 0.085(0.078) & 0.125(0.116) & 0.051 & 0.061(0.061) & 0.096(0.095) & 0.125(0.124)        \\ \hline
                  $\tilde{\mathcal{L}}$& 0.061 & 0.098(0.082) & 0.328(0.299) & 0.546(0.518) & 0.056 & 0.096(0.088) & 0.353(0.330) & 0.606(0.579)        \\ \hline
                  $\mathcal{K}$& 0.064 & \textbf{0.155}(0.143) & \textbf{0.626}(0.599) & \textbf{0.882}(0.860) & 0.054 & \textbf{0.166}(0.161) & \textbf{0.666}(0.654) & \textbf{0.914}(0.910)       \\ \hline
\multirow{2}{*}{
} & \multicolumn{4}{c}{$(n_{1},n_{2},p)=(200,144,180)$} & \multicolumn{4}{c}{$(n_{1},n_{2},p)=(400,288,360)$}       \\ \cline{2-9} 
                  & a=0 & a=3 & a=7 & a=10 & a=0 & a=3 & a=7 & a=10     \\ \hline
                  $\mathcal{L}$& 0.058 & 0.078(0.073) & 0.087(0.081) & 0.114(0.102) & 0.066 & 0.073(0.055) & 0.094(0.079) & 0.141(0.118)      \\ \hline
                  $\tilde{\mathcal{L}}$& 0.046 & 0.111(0.121) & 0.353(0.381) & 0.584(0.606) & 0.059 & 0.118(0.103) & 0.363(0.324) & 0.624(0.578)      \\ \hline
                  $\mathcal{K}$&  0.044 & \textbf{0.193}(0.213) & \textbf{0.684}(0.701) & \textbf{0.935}(0.939) & 0.039 & \textbf{0.179}(0.194) & \textbf{0.720}(0.744) & \textbf{0.934}(0.943)      \\ \hline
\multirow{2}{*}{\begin{tabular}[c]{@{}c@{}}\thead{ (iv) }\end{tabular}} & \multicolumn{4}{c}{$(n_{1},n_{2},p)=(50,50,45)$} & \multicolumn{4}{c}{$(n_{1},n_{2},p)=(100,100,90)$}       \\  \cline{2-9} 
                  & a=0 & a=3 & a=7 & a=10 & a=0 & a=3 & a=7 & a=10      \\ \hline
                  $\mathcal{L}$& 0.055 & 0.046(0.044) & 0.058(0.051) & 0.068(0.066) & 0.064 & 0.064(0.052) & 0.059(0.051) & 0.058(0.045)        \\ \hline
                  $\tilde{\mathcal{L}}$& 0.059 & 0.122(0.103) & 0.414(0.382) & 0.649(0.624) & 0.048 & 0.110(0.114) & 0.451(0.456) & 0.687(0.694)        \\ \hline
                  $\mathcal{K}$& 0.051 & \textbf{0.257}(0.256) & \textbf{0.777}(0.774) & \textbf{0.980}(0.979) & 0.039 & \textbf{0.234}(0.269) & \textbf{0.826}(0.843) & \textbf{0.983}(0.986)        \\ \hline
\multirow{2}{*}{
} & \multicolumn{4}{c}{$(n_{1},n_{2},p)=(200,200,180)$} & \multicolumn{4}{c}{$(n_{1},n_{2},p)=(400,400,360)$}       \\  \cline{2-9} 
                  & a=0 & a=3 & a=7 & a=10 & a=0 & a=3 & a=7 & a=10      \\ \hline
                  $\mathcal{L}$& 0.049 & 0.053(0.054) & 0.045(0.045) & 0.057(0.058) & 0.058 & 0.053(0.043) & 0.043(0.038) & 0.054(0.045)      \\ \hline
                  $\tilde{\mathcal{L}}$& 0.050 & 0.126(0.128) & 0.451(0.461) & 0.758(0.761) & 0.042 & 0.102(0.116) & 0.455(0.477) & 0.753(0.769)      \\ \hline
                  $\mathcal{K}$& 0.034 & \textbf{0.256}(0.313) & \textbf{0.828}(0.874) & \textbf{0.992}(0.995) & 0.048 & \textbf{0.251}(0.255) & \textbf{0.864}(0.867) & \textbf{0.988}(0.988)      \\ \hline
\end{tabular}%
}
\caption{Case 4: Empirical sizes and powers. The results are based on the 5\% significance level. The power number in small bracket indicates the size corrected power based empirical 95\% quantile not asymptotic. (i) indicates the case of $y_{1}>1,y_{2}>1$, (ii) does $y_{1}>1,y_{2}<1$, (iii) does $y_{1}<1,y_{2}>1$, and (iv) does $y_{1}<1,y_{2}<1$.}
\label{tab:table4}
\end{table}

\begin{figure} [htb!]
\centering
\caption{Power divergence in Case 1.T1,T2 stand for $\mathcal{L}$ and $\tilde{\mathcal{L}}$ respectively. $K$ 
stands for $\mathcal{K}$.}
\begin{tabular}{cccc}
\includegraphics[width=.4\textwidth, height=0.28\textwidth]{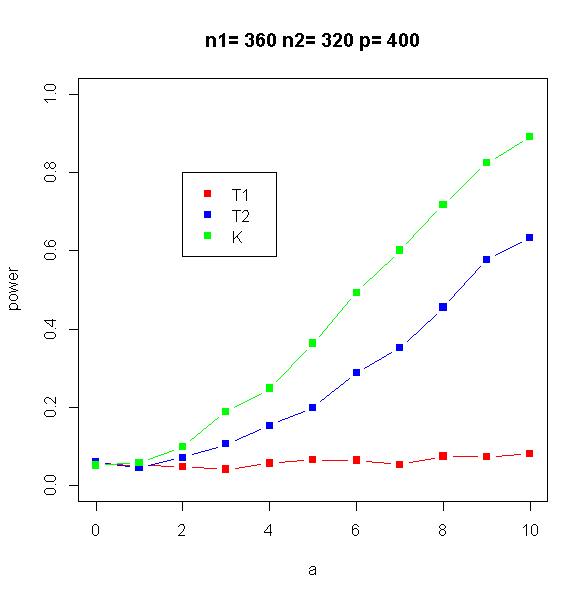} &
\includegraphics[width=.4\textwidth, height=0.3\textwidth]{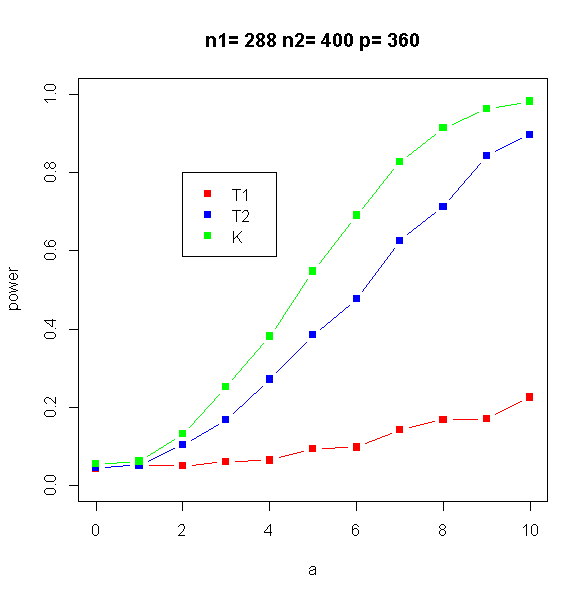}  \\
\end{tabular}
\begin{tabular}{cccc}
\includegraphics[width=.4\textwidth, height=0.28\textwidth]{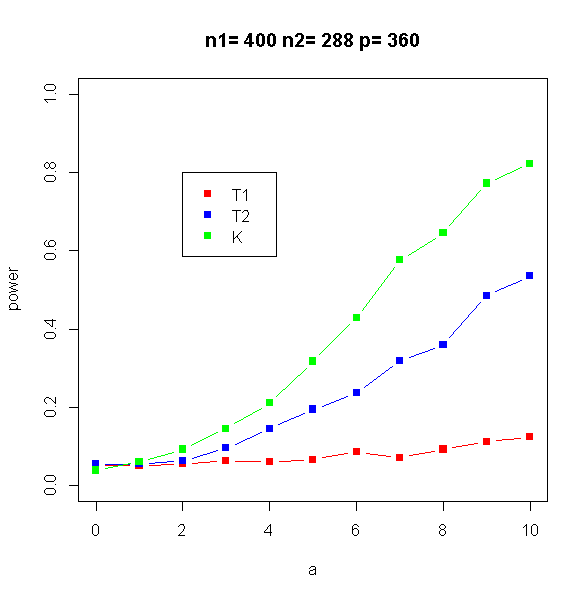} &
\includegraphics[width=.4\textwidth, height=0.28\textwidth]{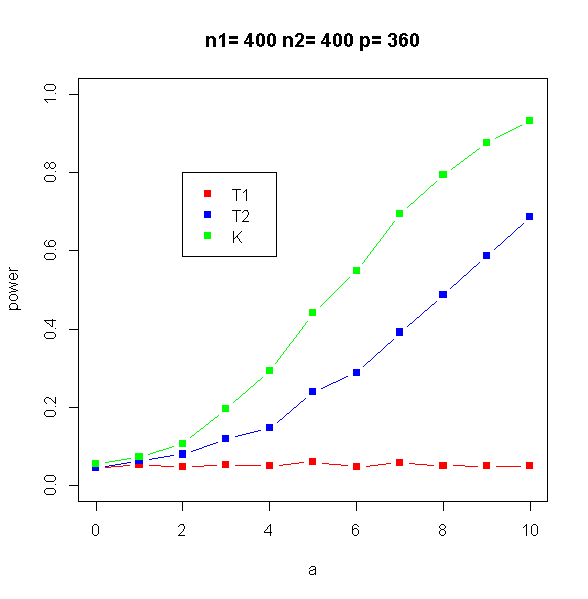} \\
\end{tabular}
\label{fig:power1}
\caption{Power divergence in Case 2. T1,T2 stand for $\mathcal{L}$ and $\tilde{\mathcal{L}}$ respectively. $K$ 
stands for $\mathcal{K}$.}
\begin{tabular}{cccc}
\includegraphics[width=.4\textwidth, height=0.28\textwidth]{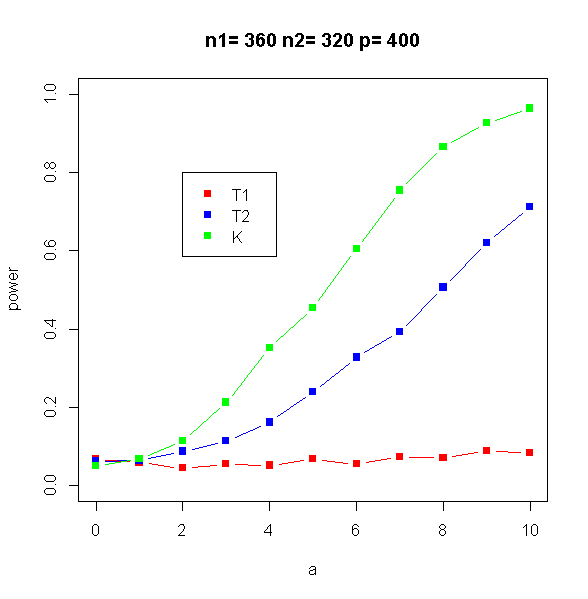} &
\includegraphics[width=.4\textwidth, height=0.28\textwidth]{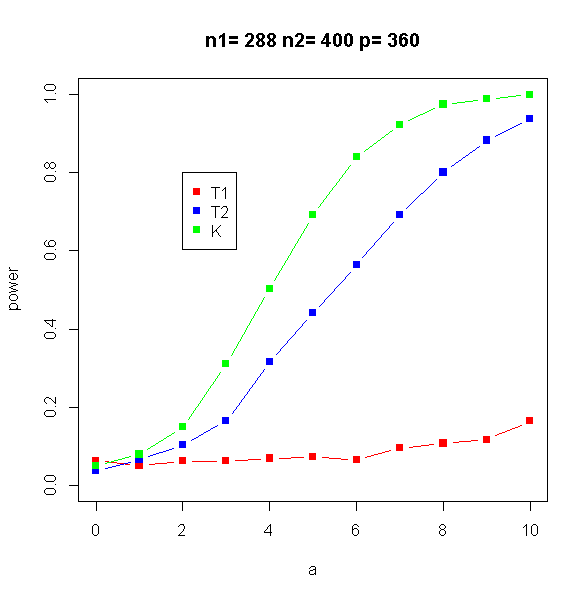}  \\
\end{tabular}
\begin{tabular}{cccc}
\includegraphics[width=.4\textwidth, height=0.28\textwidth]{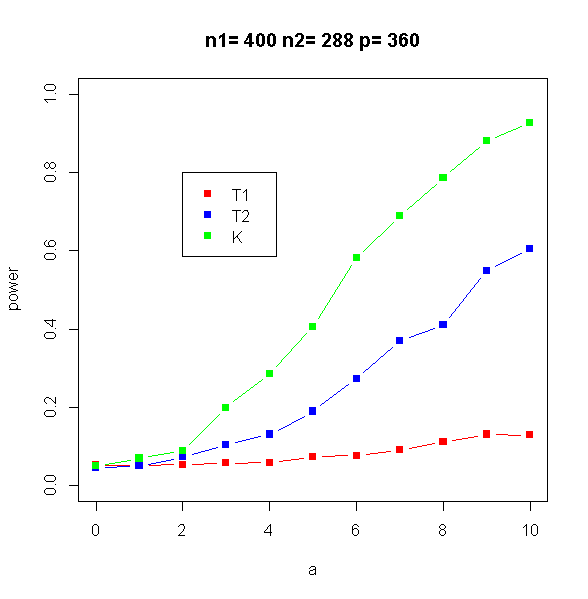} &
\includegraphics[width=.4\textwidth, height=0.29\textwidth]{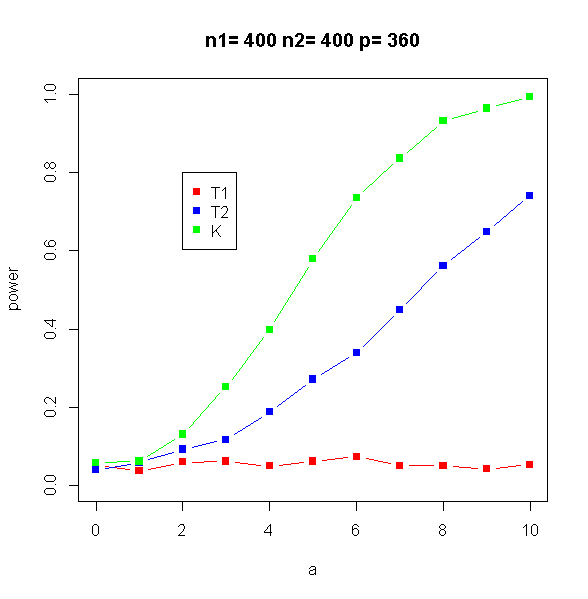} \\
\end{tabular}
\label{fig:power2}
\end{figure}

\begin{figure} [htb!]
\centering
\caption{Power divergence in Case 1.T1,T2 stand for $\mathcal{L}$ and $\tilde{\mathcal{L}}$ respectively. $K$ 
stands for $\mathcal{K}$.}
\begin{tabular}{cccc}
\includegraphics[width=.4\textwidth, height=0.28\textwidth]{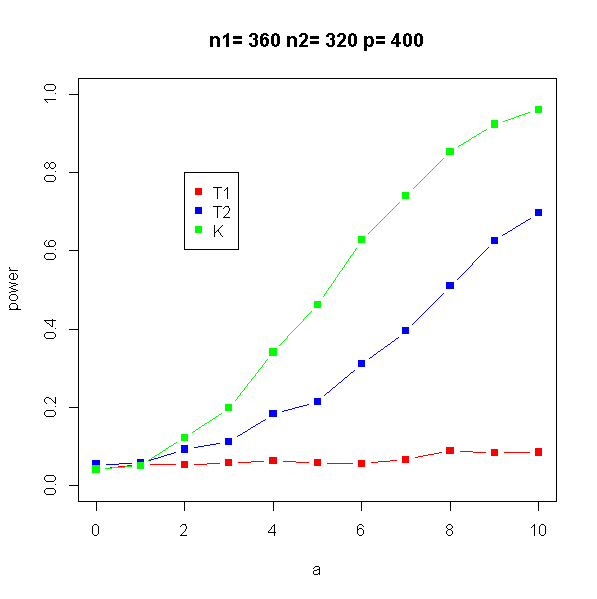} &
\includegraphics[width=.4\textwidth, height=0.28\textwidth]{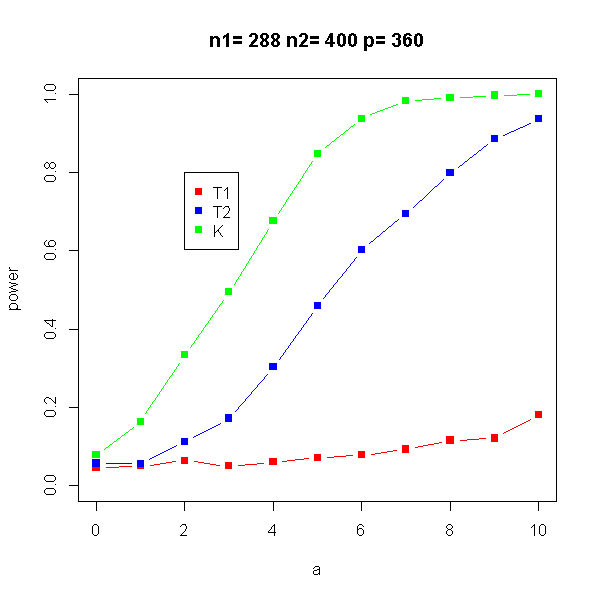}  \\
\end{tabular}
\begin{tabular}{cccc}
\includegraphics[width=.4\textwidth, height=0.28\textwidth]{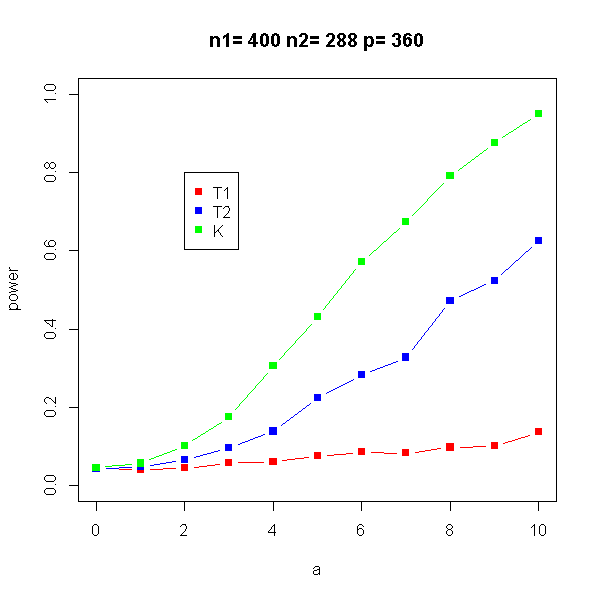} &
\includegraphics[width=.4\textwidth, height=0.28\textwidth]{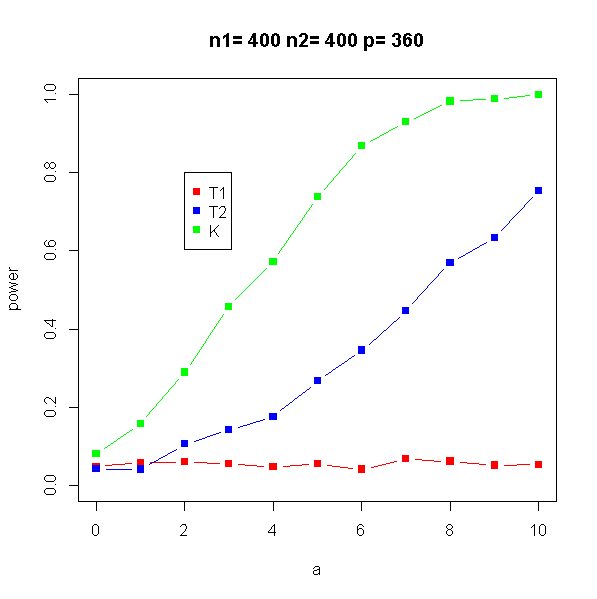} \\
\end{tabular}
\label{fig:power3}
\caption{Power divergence in Case 1.T1,T2 stand for $\mathcal{L}$ and $\tilde{\mathcal{L}}$ respectively. $K$ 
stands for $\mathcal{K}$.}
\begin{tabular}{cccc}
\includegraphics[width=.4\textwidth, height=0.28\textwidth]{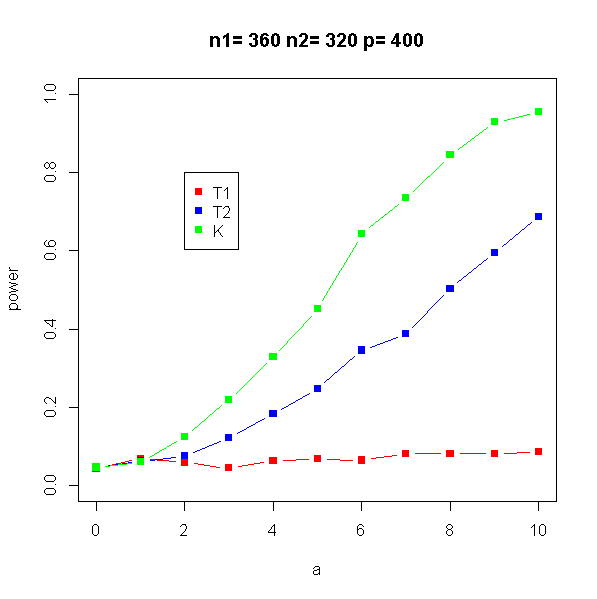} &
\includegraphics[width=.4\textwidth,height=0.28\textwidth]{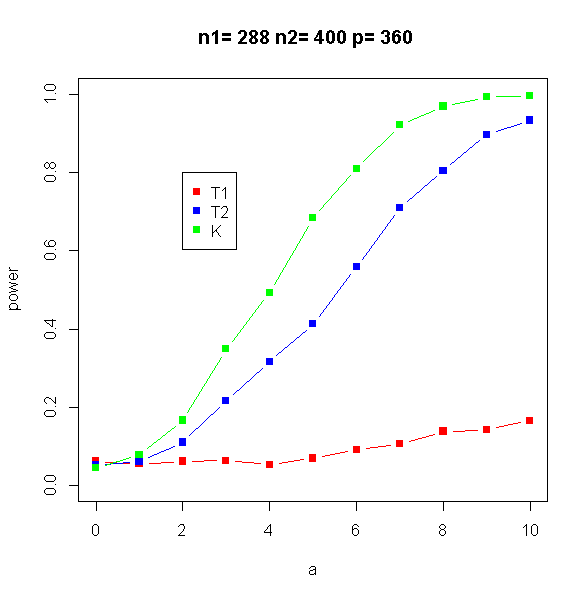}  \\
\end{tabular}
\begin{tabular}{cccc}
\includegraphics[width=.4\textwidth, height=0.28\textwidth]{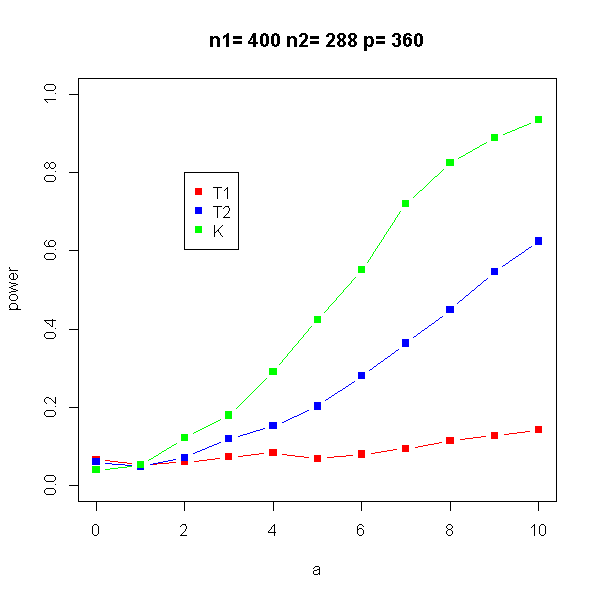} &
\includegraphics[width=.4\textwidth, height=0.28\textwidth]{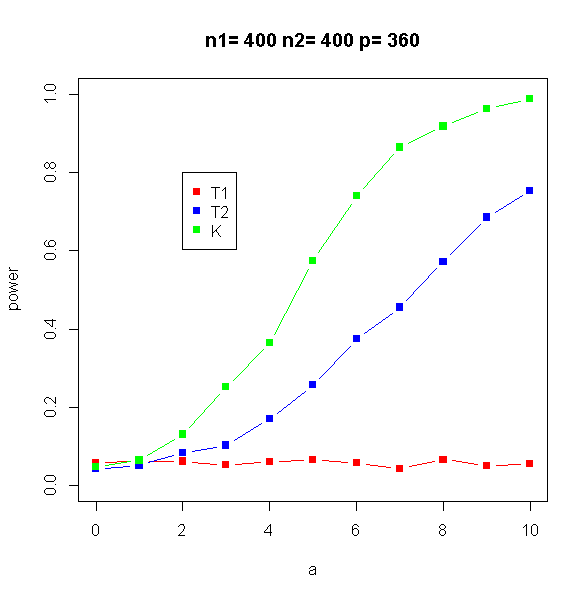} \\
\end{tabular}
\label{fig:power4}
\end{figure}

\clearpage 

\section{Conclusion} 

In this paper, we suggest an alternative invariant test, named as $\mathcal{K}$, to two modified LRTs by \citet{Zhang:2019} 
for  the equality of two large scale covariance matrices. It is based on  the sum of eigenvalues of $\mathbf{B}_{n}=n_{1}S_{n_{1}}(n_{1}S_{n_{1}}+n_{2}S_{n_{2}})^{-1}$. We find the asymptotic null distribution of  $\mathcal{K}$, when all $n_1,n_2,p$ approach $\infty$ 
at the same rate. The numerical study shows the new invariant test is more powerful than the modified LRTs by \citet{Zhang:2019}
in all cases we consider. However, we do not claim the proposed $\mathcal{K}$ is the most powerful for the problem because, 
as we learn from lower dimensional cases, there could be more powerful one than $\mathcal{K}$ for some settings. 

\section*{Acknowledgement} 

This research is supported by National Research Foundation of Korea (Grant Number: NRF-2017R1A2B
2012264).

\section*{Appendix: Proof of Theorem~\ref{thm:1}}

We will use the Cauchy's residue theorem and Lemma A.1. in \citet{Zhang:2019}.

\begin{lem}[\citealt*{Zhang:2019}, Lemma A.1]\label{lemma:1}
In addition to the Moments Assumption and the Dimensions Assumption, we further assume that:
\begin{itemize}
\item[(1)] as $\min\{p,n_{1},n_{2}\} \rightarrow \infty$, $y_{1} \rightarrow \gamma_{1} \in (0,1) \cup (1,\infty)$, $y_{2} \rightarrow \gamma_{2} \in (0,1) \cup (1,\infty)$, and $\alpha_{n} \rightarrow \alpha>0$.
\item[(2)] let $f_{1},\ldots,f_{k}$ be the analytic functions on an open region containing the interval $[a_{l},a_{r}]$, where $a_{l}=v^{-1}(1-\sqrt{\gamma_{1}})^2$, $a_{r}=1-\alpha v^{-1}(1-\sqrt{\gamma_{2}})^2$, and $v$ is defined as $v=\left(1+\frac{\gamma_{1}}{\gamma_{2}}\right)\left(1-\sqrt{\frac{\gamma_{1}\gamma_{2}}{\gamma_{1}+\gamma_{2}}}\right)^{2}$.
\end{itemize}
Then, as $\min\{p,n_{1},n_{2}\} \rightarrow \infty$, the random vector
\begin{equation} \nonumber
\left(\int f_{i} dG_{n} \right)_{i=1,\ldots,k}
\end{equation}
converges weakly to a Gaussian vector $\left(G_{f_{1}},\ldots,G_{f_{k}}\right)$ with mean functions
\begin{eqnarray}
& & \mathbb{E}G_{f_{i}} = \frac{1}{4\pi i }\oint f_{i}\left(\frac{z}{\alpha+z}\right) d\log\left(\frac{(1-\gamma_{2})m_{3}^{2}(z)+2m_{3}(z)+1-\gamma_{1}}{\left(1+m_{3}(z)\right)^2}\right) \nonumber\\
& & \quad +\frac{\Delta_{1}}{2\pi i}\oint\gamma_{1}f_{i}\left(\frac{z}{\alpha+z}\right)(1+m_{3}(z))^{-3}d m_{3}(z) \nonumber\\
& & \quad +\frac{\Delta_{2}}{4\pi i}\oint f_{i}\left(\frac{z}{\alpha+z}\right)\left(1-\frac{\gamma_{2}m_{3}^{2}(z)}{\left(1+m_{3}(z)\right)^{2}}\right)d\log\left(1-\frac{\gamma_{2}m_{3}^{2}(z)}{\left(1+m_{3}(z)\right)^{2}}\right) \label{eq:mean_part}
\end{eqnarray}
and covariance functions
\begin{eqnarray}
& & {\rm Cov}(G_{f_{i}},G_{f_{j}}) = -\frac{1}{2\pi^{2}}\oint\oint\frac{f_{i}\left(\frac{z_{1}}{\alpha+z_{1}}\right)f_{j}\left(\frac{z_{2}}{\alpha+z_{2}}\right)dm_{3}(z_{1})dm_{3}(z_{2})}{(m_{3}(z_{1})-m_{3}(z_{2}))^2} \nonumber\\
& & \quad -\frac{\gamma_{1}\Delta_{1}+\gamma_{2}\Delta_{2}}{4\pi^{2}}\oint\oint\frac{f_{i}\left(\frac{z_{1}}{\alpha+z_{1}}\right)f_{j}\left(\frac{z_{2}}{\alpha+z_{2}}\right)dm_{3}(z_{1})dm_{3}(z_{2})}{(m_{3}(z_{1})+1)^2(m_{3}(z_{2})+1)^2} \label{eq:cov_part}
\end{eqnarray}
where
\begin{eqnarray}
m_{0}(z) & = & \frac{(1+\gamma_{1})(1-z)-\alpha z(1-\gamma_{2})}{2z(1-z)(\gamma_{1}(1-z)+\alpha z\gamma_{2})} \nonumber\\
&& \quad +\frac{\sqrt{\left( (1-\gamma_{1})(1-z)+\alpha z(1-\gamma_{2}) \right)^2-4\alpha z(1-z)}}{2z(1-z)(\gamma_{1}(1-z)+\alpha z\gamma_{2})}-\frac{1}{z}, \nonumber\\
m_{1}(z) & = & \frac{\alpha}{(\alpha+z)^2}m_{0}\left(\frac{z}{\alpha +z}\right)-\frac{1}{\alpha+z}, \nonumber\\
m_{2}(z) & = & -\frac{1-\gamma_{1}}{z}+\gamma_{1}m_{1}(z), \nonumber\\
m_{mp}^{\gamma_{2}}(z) & = & \frac{1-\gamma_{2}-z+\sqrt{(z-1-\gamma_{2})^2-4\gamma_{2}}}{2\gamma_{2}z}, \nonumber\\
m_{3}(z) & = & \gamma_{2}m_{mp}^{\gamma_{2}}(-m_{2}(z))+ \frac{1-\gamma_{2}}{m_{2}(z)}. \nonumber
\end{eqnarray}
All of the above contour integrals can be evaluated on any contour enclosing the interval $\left[\frac{\alpha c_{l}}{1-c_{l}},\frac{\alpha c_{r}}{1-c_{r}}\right]$.
\end{lem}

\begin{proof}[Proof of Theorem \ref{thm:1}]
The proof of Theorem \ref{thm:1} consists of three parts: the proof of the limit part $l_{n}$, the proof of the limit part $\mu_{n}$, and the proof of the limit part $\sigma_{n}^{2}$.

\begin{proof}[Proof of the limit part $l_{n}$]
We need to calculate the following integrals:
\begin{eqnarray}
& & \int_{x_{l}}^{x_{r}} x\frac{(\alpha_{n}+1)\sqrt{(x_{r}-x)(x-x_{l})}}{2\pi y_{1}x(1-x)} dx \quad \mbox{and} \label{eq:l_n1}\\
& & \int_{x_{l}}^{x_{r}} (1-x)\frac{(\alpha_{n}+1)\sqrt{(x_{r}-x)(x-x_{l})}}{2\pi y_{1}x(1-x)} dx. \label{eq:l_n2}
\end{eqnarray}
To get these values, we choose a transformation
\begin{equation} \label{eq:transformation}
x=\frac{y_{2} \left\vert y_{1}+h\xi \right\vert^{2}}{(y_{1}+y_{2})^{2}}.
\end{equation}
We can see that the integral \eqref{eq:l_n1} is rewritten as
\begin{equation} \nonumber
\eqref{eq:l_n1} = \frac{y_{2}h^{2}i}{4\pi (y_{1}+y_{2})}\oint_{|\xi|=1}\frac{(\xi^2-1)^2}{\xi^3|y_{2}-h\xi|^2}d\xi.
\end{equation}
The singularities of the integrand are 0, $\frac{y_{2}}{h}$, and $\frac{h}{y_{2}}$. If $y_{2}>1$, then $y_{2}^2-h^2=(y_{1}+y_{2})(y_{2}-1)>0$, so $y_{2}>h$. It implies that if $y_{2}>1$, $\frac{y_{2}}{h}$ is not a pole, thus $0$ and $\frac{h}{y_{2}}$ are only poles of the integrand. On the other hand, if $y_{2} < 1$, then $y_{2} < h$ and thus 0 and $\frac{y_{2}}{h}$ are only poles of the integrand. First, we assume $y_{2}>1$. Using Cauchy's residue theorem, we get the following result:
\begin{eqnarray}
\eqref{eq:l_n1} & = & \frac{y_{2}h^{2}i}{4\pi (y_{1}+y_{2})} 2\pi i\left[ \left.\frac{d}{d\xi}\frac{(\xi^{2} - 1)^{2}}{(y_{2}-h\xi)( y_{2}\xi - h)}\right\vert_{\xi = 0} + \left.\frac{(\xi^{2} - 1)^{2}}{y_{2}\xi^{2}(y_{2}-h\xi)}\right\vert_{\xi = \frac{h}{y_{2}}} \right] \nonumber\\
& = & \frac{h^2}{y_{2}(y_{1}+y_{2})}. \nonumber
\end{eqnarray}
Next, we assume $y_{2} < 1$. Using Cauchy's residue theorem again, we can calculate the integral \eqref{eq:l_n1} as
\begin{eqnarray}
\eqref{eq:l_n1} & = & \frac{y_{2}h^{2}i}{4\pi (y_{1}+y_{2})} \left[ \left.\frac{d}{d\xi}\frac{(\xi^{2} - 1)^{2}}{(y_{2}-h\xi)( y_{2}\xi - h)}\right\vert_{\xi = 0} - \left.\frac{(\xi^{2} - 1)^{2}}{h\xi^{2}(y_{2}\xi-h)}\right\vert_{\xi = \frac{y_{2}}{h}} \right] \nonumber\\
& = & \frac{y_{2}}{y_{1}+y_{2}}. \nonumber
\end{eqnarray}
Now, we consider the integral \eqref{eq:l_n2}. Using the transformation \eqref{eq:transformation}, the integral \eqref{eq:l_n2} is rewritten as
\begin{equation} \nonumber
\eqref{eq:l_n2} = \frac{y_{1}h^{2}i}{4\pi (y_{1}+y_{2})}\oint_{|\xi|=1}\frac{(\xi^2-1)^2}{\xi^3|y_{1}+h\xi|^2}d\xi.
\end{equation}
We note that the singularities of the integrand of \eqref{eq:l_n2} are 0, $-\frac{y_{1}}{h}$, and $-\frac{h}{y_{1}}$ of which absolute values are greater than or less than 1 depending only on the size of $y_{1}$. It implies that the value of \eqref{eq:l_n2} does not depend on the size of $y_{2}$. Since $\int_{x_{l}}^{x_{r}} \frac{(\alpha_{n}+1)\sqrt{(x_{r}-x)(x-x_{l})}}{2\pi y_{1}x(1-x)} dx = \frac{h^2}{y_{1}y_{2}}$ for $y_{1},y_{2} > 1$, we can easily calculate the value of \eqref{eq:l_n2} when $y_{1} > 1$ as
\begin{equation} \nonumber
\eqref{eq:l_n2} = \frac{h^2}{y_{1}y_{2}} - \frac{h^2}{y_{2}(y_{1}+y_{2})} = \frac{h^2}{y_{1}(y_{1}+y_{2})}.
\end{equation}
Similarly, if $y_{1}<1$, then
\begin{equation} \nonumber
\eqref{eq:l_n2} = 1 - \frac{y_{2}}{y_{1}+y_{2}} = \frac{y_{1}}{y_{1}+y_{2}}.
\end{equation}
Here, we used the fact that $\int_{x_{l}}^{x_{r}} \frac{(\alpha_{n}+1)\sqrt{(x_{r}-x)(x-x_{l})}}{2\pi y_{1}x(1-x)} dx = 1$ for $y_{1},y_{2} < 1$.
\end{proof}

\begin{proof}[Proof of the limit part $\mu_{n}$]

According to Lemma 5.1 in \citet{Bai:2015}, $m_{3}$ satisfies the equation
\begin{equation} \nonumber
    z=-\frac{m_{3}(m_{3}+1-y_{1})}{(1-y_{2})(m_{3}+\frac{1}{1-y_{2}})}.
\end{equation}
As in the proof of the limit part $\tilde{\mu}_{n}$ of Theorem 2.5 in \citet{Zhang:2019},
by making an integral conversion 
\begin{equation} \nonumber
    z=\frac{(1+hr\xi)(1+\frac{h}{r\xi})}{(1-y_{2})^2},
\end{equation}
where $r$ is a number greter than but close to 1, $m_{3}$ satisfies
\begin{equation} \nonumber
    \frac{(1+hr\xi)(1+\frac{h}{r\xi})}{(1-y_{2})^2}=-\frac{m_{3}(m_{3}+1-y_{1})}{(1-y_{2})m_{3}+1}.
\end{equation}
By the above equation, we can obtain $m_{3}=-\frac{1+\frac{h}{r\xi}}{1-y_{2}}$ or $m_{3}=-\frac{1+hr\xi}{1-y_{2}}$.  
When $z$ runs in the positive direction around the contour enclosing the interval $[\frac{\alpha c_{l}}{1-c_{l}},\frac{\alpha c_{r}}{1-c_{r}}]$, $m_{3}$ runs in the opposite direction. 
Thus, by substituting
\begin{equation} \nonumber
m_{3}(z)= \left\{ \begin{array}{l l}
\displaystyle -\frac{1+\frac{h}{r\xi}}{1-y_{2}}& \quad\mbox{if } y_{2}>1, \\\\
\displaystyle  -\frac{1+hr\xi}{1-y_{2}} & \quad\mbox{if } y_{2}<1,
\end{array} \right.
\end{equation}
where $r$ is a real number greater than but close to 1, we observe the followings :
\begin{eqnarray}
&& \frac{z}{\alpha+z}=\frac{\left(\xi+\frac{1}{hr}\right)\left(\xi+\frac{h}{r}\right)}{\left(\xi+\frac{y_{2}}{hr}\right)\left(\xi+\frac{h}{y_{2}r}\right)}, \nonumber\\
&& d\log\left(\frac{(1-y_{2})m_{3}^{2}(z)+2m_{3}(z)+1-y_{1}}{\left(1+m_{3}(z)\right)^2}\right) \nonumber\\
&& \qquad\qquad = \left\{ \begin{array}{l l}
\left( \frac{1}{\xi - \frac{1}{r}} + \frac{1}{\xi + \frac{1}{r}} - \frac{2}{\xi + \frac{h}{y_{2}r}} \right)d\xi & \mbox{if } y_{2} > 1, \\\\
\left( \frac{1}{\xi - \frac{1}{r}} + \frac{1}{\xi + \frac{1}{r}} - \frac{2}{\xi + \frac{y_{2}}{hr}} \right)d\xi & \mbox{if } y_{2} < 1,\end{array} \right. \nonumber\\
&& (1+m_{3}(z))^{-3}d m_{3}(z) = \left\{ \begin{array}{l l}
\frac{(1-y_{2})^{2}h}{y_{2}^{3}}\frac{1}{(\xi+\frac{y_{2}}{hr})^3}d\xi & \mbox{if } y_{2} > 1, \\\\
\frac{(1-y_{2})^2}{h^2}\frac{1}{(\xi+\frac{y_{2}}{hr})^3}d\xi & \mbox{if } y_{2} < 1,
\end{array} \right. \nonumber\\
&& \left(1-\frac{y_{2}m_{3}^{2}(z)}{(1+m_{3}(z))^{2}} \right)d\log\left(1-\frac{y_{2}m_{3}^{2}(z)}{(1+m_{3}(z))^{2}} \right) \nonumber\\
&& \qquad\qquad = \left\{ \begin{array}{l l}
\frac{(y_{2}-1)\left(\xi^2-\frac{h^{2}}{y_{2}r^2}\right)}{y_{2}\left(\xi+\frac{h}{y_{2}r}\right)^2}\left(\frac{2\xi}{\xi^2-\frac{h^{2}}{y_{2}r^2}}-\frac{2}{\xi+\frac{h}{y_{2}r}} \right)d\xi & \mbox{if } y_{2} > 1, \\\\
\frac{(1-y_{2})\left(\xi^2-\frac{y_{2}}{h^2r^2}\right)}{\left(\xi+\frac{y_{2}}{hr}\right)^2}\left(\frac{2\xi}{\xi^2-\frac{y_{2}}{h^2r^2}}-\frac{2}{\xi+\frac{y_{2}}{hr}} \right)d\xi & \mbox{if } y_{2} < 1.
\end{array} \right.\nonumber
\end{eqnarray}
First, we assume $y_{2}>1$. Then, the mean part is same with
\begin{eqnarray}
& & \mathbb{E}G_{f_{1}} = \lim_{r\downarrow 1}\frac{1}{4\pi i}\oint_{|\xi|=1}\frac{z}{\alpha+z}\left(\frac{1}{\xi-\frac{1}{r}}+\frac{1}{\xi+\frac{1}{r}}-\frac{2}{\xi+\frac{h}{y_{2}r}}\right)d\xi \label{eq:mean_1}\\
& & \quad +\lim_{r\downarrow 1}\frac{-y_{1}\Delta_{1}}{2\pi i}\oint_{|\xi|=1}\frac{z}{\alpha+z}\frac{(1-y_{2})^2h}{y_{2}^3}\frac{\xi}{(\xi+\frac{h}{y_{2}r})^3}d\xi \label{eq:mean_2}\\
& & \quad +\lim_{r\downarrow 1}\frac{\Delta_{2}}{4\pi i}\oint_{|\xi|=1}\frac{z}{\alpha+z}\frac{(y_{2}-1)(\xi^2-\frac{h^2}{y_{2}r^2})}{y_{2} (\xi+\frac{h}{y_{2}r})^2} \left(\frac{2\xi}{\xi^2-\frac{h^2}{y_{2}r^2}}-\frac{2}{\xi+\frac{h}{y_{2}r}} \right)d\xi. \label{eq:mean_3}
\end{eqnarray}
In (\ref{eq:mean_1}), the poles are $\pm\frac{1}{r}$ and $-\frac{h}{y_{2}r}$. Hence, \eqref{eq:mean_1} can be calculated as
\begin{eqnarray}
&& \eqref{eq:mean_1} = \lim_{r\downarrow 1}\frac{1}{4\pi i}\oint_{|\xi|=1} \frac{\left(\xi+\frac{1}{hr}\right)\left(\xi+\frac{h}{r}\right)}{\left(\xi + \frac{y_{2}}{hr}\right)\left(\xi+\frac{h}{y_{2}r}\right)} \left(\frac{1}{\xi-\frac{1}{r}}+\frac{1}{\xi+\frac{1}{r}}-\frac{2}{\xi+\frac{h}{y_{2}r}}\right)d\xi \nonumber\\
&& = \frac{1}{2} \left[ \left. \frac{\left(\xi+\frac{1}{hr}\right)\left(\xi+\frac{h}{r}\right)}{\left(\xi + \frac{y_{2}}{hr}\right)\left(\xi+\frac{h}{y_{2}r}\right)} \right\vert_{\xi = \frac{1}{r}} + \left. \frac{\left(\xi+\frac{1}{hr}\right)\left(\xi+\frac{h}{r}\right)}{\left(\xi + \frac{y_{2}}{hr}\right)\left(\xi+\frac{h}{y_{2}r}\right)} \right\vert_{\xi = -\frac{1}{r}} \right. \nonumber\\
&& \quad \left. + \left. \frac{\left(\xi+\frac{1}{hr}\right)\left(\xi+\frac{h}{r}\right)}{\left(\xi + \frac{y_{2}}{hr}\right)}\left( \frac{1}{\xi - \frac{1}{r}} + \frac{1}{\xi + \frac{1}{r}} \right) \right\vert_{\xi = -\frac{h}{y_{2}r}} -2 \left. \frac{d}{d\xi}\frac{\left(\xi+\frac{1}{hr}\right)\left(\xi+\frac{h}{r}\right)}{\left(\xi + \frac{y_{2}}{hr}\right)} \right\vert_{\xi = -\frac{h}{y_{2}r}} \right] \nonumber\\
&& = 0. \nonumber
\end{eqnarray}
Here, we used the relations $(y_{1}^2-h^2)=(y_{1}+y_{2})(y_{1}-1)$ and $(y_{1}-h^2)=y_{2}(y_{1}-1)$. Similarly, \eqref{eq:mean_2} and \eqref{eq:mean_3} can be calculated as follows :
\begin{eqnarray}
\eqref{eq:mean_2} & = & \lim_{r\downarrow 1}\frac{-y_{1}(1-y_{2})^{2}h\Delta_{1}}{2\pi i y_{2}^{3}}\oint_{|\xi|=1} \frac{\xi\left(\xi+\frac{1}{hr}\right)\left(\xi+\frac{h}{r}\right)}{\left(\xi+\frac{y_{2}}{hr}\right)\left(\xi+\frac{h}{y_{2}r}\right)^{4}}d\xi \nonumber\\
& = & \lim_{r\downarrow 1}\frac{-y_{1}(1-y_{2})^{2}h\Delta_{1}}{y_{2}^{3}} \left. \frac{1}{6}\frac{d^{3}}{d\xi^{3}} \frac{\xi\left(\xi+\frac{1}{hr}\right)\left(\xi+\frac{h}{r}\right)}{\left(\xi+\frac{y_{2}}{hr}\right)} \right\vert_{\xi = -\frac{h}{y_{2}r}} \nonumber\\
& = & -\frac{\Delta_{1}h^2y_{1}^2y_{2}^2}{(y_{1}+y_{2})^4}, \nonumber\\
\eqref{eq:mean_3} & = & \lim_{r\downarrow 1}\frac{\Delta_{2}}{4\pi i}\oint_{|\xi|=1} \frac{\left(\xi+\frac{1}{hr}\right)\left(\xi+\frac{h}{r}\right)}{\left(\xi+\frac{y_{2}}{hr}\right)\left(\xi+\frac{h}{y_{2}r}\right)} \frac{(y_{2}-1)(\xi^2-\frac{h^2}{y_{2}r^2})}{y_{2} (\xi+\frac{h}{y_{2}r})^2} \left(\frac{2\xi}{\xi^2-\frac{h^2}{y_{2}r^2}}-\frac{2}{\xi+\frac{h}{y_{2}r}} \right)d\xi \nonumber\\
& = & \lim_{r\downarrow 1}\frac{\Delta_{2}(y_{2}-1)}{2\pi iy_{2}} \left[ \oint_{|\xi|=1} \frac{\xi\left(\xi+\frac{1}{hr}\right)\left(\xi+\frac{h}{r}\right)}{\left(\xi+\frac{y_{2}}{hr}\right)\left(\xi+\frac{h}{y_{2}r}\right)^{3}} d\xi \right. \nonumber\\
&& \qquad\qquad\qquad\qquad\qquad\qquad\left.+ \oint_{|\xi|=1} \frac{\left(\xi+\frac{1}{hr}\right)\left(\xi+\frac{h}{r}\right)\left(\xi^2-\frac{h^2}{y_{2}r^2}\right)}{\left(\xi+\frac{y_{2}}{hr}\right)\left(\xi+\frac{h}{y_{2}r}\right)^{4}} d\xi \right] \nonumber\\
& = & \lim_{r\downarrow 1}\frac{\Delta_{2}(y_{2}-1)}{y_{2}} \left[ \frac{1}{2}\frac{d^{2}}{d\xi^{2}} \frac{\xi\left(\xi+\frac{1}{hr}\right)\left(\xi+\frac{h}{r}\right)}{\left(\xi+\frac{y_{2}}{hr}\right)} \right. \nonumber\\
&& \qquad\qquad\qquad\qquad\qquad\qquad\qquad\left.+ \frac{1}{6}\frac{d^{3}}{d\xi^{3}} \frac{\left(\xi+\frac{1}{hr}\right)\left(\xi+\frac{h}{r}\right)\left(\xi^2-\frac{h^2}{y_{2}r^2}\right)}{\left(\xi+\frac{y_{2}}{hr}\right)} \right]_{\xi = -\frac{h}{y_{2}r}} \nonumber\\
& = & \frac{\Delta_{2}h^2y_{1}^2y_{2}^2}{(y_{1}+y_{2})^4}. \nonumber
\end{eqnarray}
Therefore, when $y_{2}>1$,
\begin{equation} \nonumber
\mathbb{E}G_{f_{1}} = -\frac{\Delta_{1}h^2y_{1}^2y_{2}^2}{(y_{1}+y_{2})^4}+\frac{\Delta_{2}h^2y_{1}^2y_{2}^2}{(y_{1}+y_{2})^4}.
\end{equation}
For the case in which $y_{2}<1$, we observe
\begin{eqnarray}
& & \mathbb{E}G_{f_{1}} = \lim_{r\downarrow 1}\frac{1}{4\pi i}\oint_{|\xi|=1}\frac{z}{\alpha+z}\left(\frac{1}{\xi-\frac{1}{r}}+\frac{1}{\xi+\frac{1}{r}}-\frac{2}{\xi+\frac{y_{2}}{hr}}\right)d\xi \label{eq:mean_4}\\
& & \quad +\lim_{r\downarrow 1}\frac{-y_{1}\Delta_{1}}{2\pi i}\oint_{|\xi|=1}\frac{z}{\alpha+z}\frac{(1-y_{2})^2}{h^{2}}\frac{\xi}{(\xi+\frac{y_{2}}{hr})^3}d\xi \label{eq:mean_5}\\
& & \quad +\lim_{r\downarrow 1}\frac{\Delta_{2}}{4\pi i}\oint_{|\xi|=1}\frac{z}{\alpha+z}\frac{(1-y_{2})(\xi^2-\frac{y_{2}}{h^{2}r^2})}{(\xi+\frac{y_{2}}{hr})^2} \left(\frac{2\xi}{\xi^2-\frac{y_{2}}{h^{2}r^2}}-\frac{2}{\xi+\frac{y_{2}}{hr}} \right)d\xi. \label{eq:mean_6}
\end{eqnarray}
Through a simple calculation similar to the above, we can show
\begin{equation} \nonumber
\mathbb{E}G_{f_{1}} = -\frac{\Delta_{1}h^2y_{1}^2y_{2}^2}{(y_{1}+y_{2})^4}+\frac{\Delta_{2}h^2y_{1}^2y_{2}^2}{(y_{1}+y_{2})^4}.
\end{equation}
when $y_{2} < 1$. Therefore, it holds that
\begin{equation} \nonumber
\mu_{n}=-\frac{\Delta_{1}h^2y_{1}^2y_{2}^2}{(y_{1}+y_{2})^4}+\frac{\Delta_{2}h^2y_{1}^2y_{2}^2}{(y_{1}+y_{2})^4}.
\end{equation}
Finally, we note that according to the equation (\ref{eq:minus}) the expectation of the LSS of $\mathcal{P}_{2}$ is the opposite sign of that of $\mathcal{P}_{1}$.
\end{proof}

\begin{proof}[Proof of the limit part $\sigma_{n}^{2}$]
We need to calculate the following integral :
\begin{eqnarray}
& & {\rm Var}(G_{f_{1}}) = -\frac{1}{2\pi^{2}}\oint\oint \frac{z_{1}}{\alpha+z_{1}}\frac{z_{2}}{\alpha+z_{2}}\frac{dm_{3}(z_{1})dm_{3}(z_{2})}{(m_{3}(z_{1})-m_{3}(z_{2}))^2} \nonumber\\
& & \quad -\frac{\gamma_{1}\Delta_{1}+\gamma_{2}\Delta_{2}}{4\pi^{2}}\oint\oint \frac{z_{1}}{\alpha+z_{1}}\frac{z_{2}}{\alpha+z_{2}} \frac{dm_{3}(z_{1})dm_{3}(z_{2})}{(m_{3}(z_{1})+1)^2(m_{3}(z_{2})+1)^2} \label{eq:var}
\end{eqnarray}
First, we consider the case in which $y_{2}>1$. In this case, we observe
\begin{eqnarray}
\frac{dm_{3}(z_{1})dm_{3}(z_{2})}{(m_{3}(z_{1})-m_{3}(z_{2}))^2} & = & \frac{(1-y_{2})^2}{(\frac{h}{r_{1}\xi_{1}}-\frac{h}{r_{2}\xi_{2}})^2}\frac{hd\xi_{1}}{(1-y_{2})r_{1}\xi_{1}^2}\frac{hd\xi_{2}}{(1-y_{2})r_{2}\xi_{2}^2}=\frac{r_{1}r_{2}d\xi_{1}d\xi_{2}}{(r_{1}\xi_{1}-r_{2}\xi_{2})^2}, \nonumber\\
\frac{dm_{3}(z_{1})}{(m_{3}(z_{1})+1)^2} & = & \frac{(1-y_{2})^2}{(y_{2}+\frac{h}{r_{1}\xi_{1}})^2}\frac{hd\xi_{1}}{(1-y_{2})r_{1}\xi_{1}^2}=\frac{(1-y_{2})r_{1}hd\xi_{1}}{(y_{2}r_{1}\xi_{1}+h)^2}, \nonumber\\
\frac{dm_{3}(z_{2})}{(m_{3}(z_{2})+1)^2} & = & \frac{(1-y_{2})^2}{(y_{2}+\frac{h}{r_{2}\xi_{2}})^2}\frac{hd\xi_{2}}{(1-y_{2})r_{2}\xi_{2}^2}=\frac{(1-y_{2})r_{2}hd\xi_{2}}{(y_{2}r_{2}\xi_{2}+h)^2}. \nonumber
\end{eqnarray}
Assuming $r_{1} < r_{2}$ without loss of generality, the first term of \eqref{eq:var} can calculated as follows.
\begin{eqnarray}
&& -\frac{1}{2\pi^{2}}\oint\oint \frac{z_{1}}{\alpha+z_{1}}\frac{z_{2}}{\alpha+z_{2}}\frac{dm_{3}(z_{1})dm_{3}(z_{2})}{(m_{3}(z_{1})-m_{3}(z_{2}))^2} \nonumber\\
&& \quad = \lim_{r_{2} \downarrow 1} \lim_{r_{1} \downarrow 1} -\frac{1}{2\pi^{2}} \oint_{\left\vert \xi_{2} \right\vert = 1} \frac{\left(\xi_{2}+\frac{1}{hr_{2}}\right)\left(\xi_{2}+\frac{h}{r_{2}}\right)}{\left(\xi_{2}+\frac{y_{2}}{hr_{2}}\right)\left(\xi_{2}+\frac{h}{y_{2}r_{2}}\right)} \nonumber\\
&& \qquad\qquad \left[ \oint_{\left\vert \xi_{1} \right\vert = 1}  \frac{\left(\xi_{1}+\frac{1}{hr_{1}}\right)\left(\xi_{1}+\frac{h}{r_{1}}\right)}{\left(\xi_{1}+\frac{y_{2}}{hr_{1}}\right)\left(\xi_{1}+\frac{h}{y_{2}r_{1}}\right)} \frac{r_{1}r_{2}}{(r_{1}\xi_{1}-r_{2}\xi_{2})^2}d\xi_{1} \right] d\xi_{2}\nonumber\\
&& \quad = \lim_{r_{2} \downarrow 1} \frac{1}{\pi i} \oint_{\left\vert \xi_{2} \right\vert = 1} \frac{\left(\xi_{2}+\frac{1}{hr_{2}}\right)\left(\xi_{2}+\frac{h}{r_{2}}\right)}{\left(\xi_{2}+\frac{y_{2}}{hr_{2}}\right)\left(\xi_{2}+\frac{h}{y_{2}r_{2}}\right)} \frac{h(y_{2}-h^{2})(y_{2}-1)}{y_{2}r_{2}(y_{2}^{2}-h^{2})\left( \xi_{2} + \frac{h}{y_{2}r_{2}}\right)} d\xi_{2} \nonumber\\
&& \quad = \frac{2h^2y_{1}^2y_{2}^2}{(y_{1}+y_{2})^4} \nonumber
\end{eqnarray}
The second term of \eqref{eq:var} can be expressed as
\begin{eqnarray}
&& -\frac{y_{1}\Delta_{1}+y_{2}\Delta_{2}}{4\pi^{2}}\oint\oint \frac{z_{1}}{\alpha+z_{1}}\frac{z_{2}}{\alpha+z_{2}} \frac{dm_{3}(z_{1})dm_{3}(z_{2})}{(m_{3}(z_{1})+1)^2(m_{3}(z_{2})+1)^2} \nonumber\\
&& \quad = -\frac{y_{1}\Delta_{1}+y_{2}\Delta_{2}}{4\pi^{2}}\oint \frac{z_{1}}{\alpha+z_{1}} \frac{dm_{3}(z_{1})}{(m_{3}(z_{1})+1)^2} \times \oint \frac{z_{2}}{\alpha+z_{2}} \frac{dm_{3}(z_{2})}{(m_{3}(z_{2})+1)^2} \nonumber\\
&& \quad = (y_{1}\Delta_{1}+y_{2}\Delta_{2}) \left( -\frac{h^{2}y_{1}y_{2}}{(y_{1}+y_{2})^{3}} \right)^{2} = (y_{1}\Delta_{1}+y_{2}\Delta_{2}) \frac{h^{4}y_{1}^{2}y_{2}^{2}}{(y_{1}+y_{2})^{6}} \nonumber
\end{eqnarray}
Therefore, we get the desired result :
\begin{equation} \nonumber
\sigma_{n}^{2} = \frac{2h^2y_{1}^2y_{2}^2}{(y_{1}+y_{2})^4}+ (y_{1}\Delta_{1}+y_{2}\Delta_{2})\frac{h^4y_{1}^2y_{2}^2}{(y_{1}+y_{2})^6}
\end{equation}
When $y_{2}<1$, by the symmetry of $\frac{dm_{3}(z_{1})dm_{3}(z_{2})}{(m_{3}(z_{1})-m_{3}(z_{2}))^2}$, it is the same as the previous case and
\begin{eqnarray}
\frac{dm_{3}(z_{1})}{(m_{3}(z_{1})+1)^2} & = & \frac{(1-y_{2})^2}{(y_{2}+hr_{1}\xi_{1})^2}\frac{hr_{1}d\xi_{1}}{y_{2}-1}=\frac{(y_{2}-1)r_{1}hd\xi_{1}}{(hr_{1}\xi_{1}+y_{2})^2} \nonumber\\
\frac{dm_{3}(z_{2})}{(m_{3}(z_{2})+1)^2} & = & \frac{(1-y_{2})^2}{(y_{2}+hr_{2}\xi_{2})^2}\frac{hr_{2}d\xi_{2}}{y_{2}-1}=\frac{(y_{2}-1)r_{2}hd\xi_{2}}{(hr_{2}\xi_{2}+y_{2})^2}. \nonumber
\end{eqnarray}
Using the above relations, one can easily show that the variance of the LSS is equal to that in the case in which $y_{2}>1$.
\end{proof}

These complete the proof of Theorem \ref{thm:1}.
\end{proof}

\end{document}